\documentclass{nhm}
\usepackage{amsmath}
\usepackage{amsthm}
\usepackage{amssymb}  
\usepackage{amsfonts}
\usepackage{paralist}
\usepackage{graphics} 
\usepackage{epsfig} 
\usepackage{epstopdf} 
\usepackage[colorlinks=true]{hyperref}
\hypersetup{urlcolor=blue, citecolor=red}
\usepackage{graphicx, tikz}
\usepackage[english]{babel}  
\usepackage{wrapfig}
\DeclareGraphicsExtensions{.pdf,.png,.jpg}
\usepackage{epsfig} 
\usepackage{float}
\usepackage{arydshln}
\usepackage{enumitem}
\usepackage[pagewise]{lineno}

\newtheorem{theorem}{Theorem}[section]

\newtheorem{lemma}[theorem]{Lemma}
\newtheorem{proposition}{Proposition}

\theoremstyle{definition}
\newtheorem{example}{Example}[section]

\newtheorem{remark}{Remark}

\newtheorem{algorithm}{Algorithm}
\newsavebox{\mycases}

\def\R{\mathbb{R}}
\DeclareMathOperator{\diag}{diag}
\DeclareMathOperator{\rank}{rank}

\tikzstyle{metab} = [circle, text centered, draw=black, fill=gray!20]
\tikzstyle{arrow} = [thick,->,>=stealth]
\tikzstyle{vecArrow} = [thin]
\usetikzlibrary{shapes.geometric, arrows, decorations.markings}
\tikzstyle{metab1} = [circle, minimum width=2cm, minimum height=2cm,text centered, draw=black, fill=blue!30]
\tikzstyle{arrow} = [thick,->,>=stealth]
\tikzstyle{vecArrow} = [thin]
    \def\ppages{X--XX}
     \def\DOI{10.3934/nhm.xx.xx.xx}

\title[Stability via Linear-In-Flux-Expression]
      {Stability of metabolic networks via Linear-in-Flux-Expressions}

\author[Merrill, An, McQuade, Garin, Azer, Abrams and Piccoli]{}

 \email{nathaniel.j.merrill@rutgers.edu}
 \email{zheming.an@rutgers.edu}
 \email{sean.mcquade@rutgers.edu}
 \email{federica.garin@inria.fr}
 \email{karim.azer@gatesmri.org}
 \email{ruth.abrams@sanofi.com}
 \email{piccoli@camden.rutgers.edu}

\thanks{$^*$ Corresponding author: Benedetto Piccoli.}

\begin{document}
\maketitle


\centerline{\scshape Nathaniel J. Merrill, Zheming An and Sean T. McQuade}
\medskip
{\footnotesize
 \centerline{ Center for Computational and Integrative Biology}
   \centerline{Rutgers Camden. Camden NJ, USA}
}

\medskip

\centerline{\scshape Federica Garin}
\medskip
{\footnotesize
 \centerline{Researcher}
   \centerline{Univ.\ Grenoble Alpes, Inria, CNRS, Grenoble INP}
   \centerline{GIPSA-lab. Grenoble, France}
} 

\medskip

\centerline{\scshape Karim Azer}
\medskip
{\footnotesize
  \centerline{ Bill \& Melinda Gates Medical Research Institute} 
  \centerline{245 Main Street Kendall Square. Cambridege, MA 02142}
}

\medskip

\centerline{\scshape Ruth E. Abrams}
\medskip
{\footnotesize
 \centerline{ Senior scientist, Translational Informatics}
   \centerline{ Sanofi. Bridgewater NJ, USA}
}

\medskip

\centerline{\scshape Benedetto Piccoli$^*$}
\medskip
{\footnotesize
 \centerline{Joseph and Loretta Lopez chair professor of Mathematics}
   \centerline{Center for Computational and Integrative Biology}
   \centerline{Rutgers Camden. Camden NJ, USA}
} 



\begin{abstract}
The methodology named LIFE (Linear-in-Flux-Expressions) was developed with the purpose of simulating and analyzing large metabolic systems. With LIFE, the number of model parameters is reduced by accounting for correlations among the parameters of the system.  Perturbation analysis on LIFE systems results in less overall variability of the system, leading to results that more closely resemble empirical data.  These systems can be associated to graphs, and characteristics of the graph give insight into the dynamics of the system.

This work addresses two main problems: 1. for fixed metabolite levels, find all fluxes for which the metabolite levels are an equilibrium, and 2. for fixed fluxes, find all metabolite levels which are equilibria for the system. We characterize the set of solutions for both problems, and show general results relating stability of systems to the structure of the
associated graph. We show that there is a structure of the graph necessary for stable dynamics.
Along with these general results, we show how stability analysis from the fields of network flows, compartmental systems, control theory and Markov chains apply to LIFE systems.
\end{abstract}

\section{Introduction}
Quantitative Systems Pharmacology (QSP) aims to gain more information about a potential drug treatment on a human patient before the more expensive stages of development begin \cite{perez-nueno}. QSP models allow us to perform in silico experiments on a simulated metabolic system that predicts the response of perturbing a flux. A drug may be metabolized differently by various patients, and modelers working in pharmacology must anticipate these differences.  Building a profile of how the drug affects different classes of simulated patients will help the developers of new drugs understand the viability of a treatment and acquire insight into the mechanisms by which the drug acts.

A recent advancement in QSP modeling called Linear-in-Flux-Expressions (LIFE) is a method of analyzing systems of Ordinary Differential Equations (ODEs) \cite{LIFE,ACC}.  Originally, LIFE was designed to analyze metabolic systems, which are composed of Fluxes and Metabolites.  Fluxes in the metabolic system are the rates of chemical reactions in the human body, and they determine the dynamics on the metabolites, which are the various chemical compounds involved in metabolism. Modeling these systems depends on choosing fluxes, which are difficult to measure directly, so that the system effectively simulates human metabolism.

To implement LIFE on a metabolic network, the network must be written as a directed graph \cite{ACC}.   The edges of the graph represent the reaction rates (\emph{fluxes}), and the vertices represent quantities of chemical compounds (\emph{metabolites}). From the graph we construct the stoichiometric matrix of the system.  This stoichiometric matrix is not the classical one mentioned by \cite{palsson,van}.  The LIFE method is also different from QSP models whose dynamics traditionally depend on a matrix containing information about the flux of the system.  In these classical QSP models the dynamics of the metabolites are linear with respect to the metabolites.  By contrast, systems using the LIFE method are linear in fluxes and have a stoichiometric matrix that is dependent on the metabolites.

Initially, the LIFE method was developed using the human cholesterol metabolism network \cite{LIFE}. LIFE enables us to simply describe the correlations among the fluxes of the model at steady state.  There are generally many correlations among fluxes, and maintaining these correlations leads to a more consistent response to perturbing the fluxes in the system.  This was advantageous to QSP modelers, who previously analyzed flux perturbations with little to no consideration to relationships among fluxes \cite{allen}. Now, we expand our study of these systems, showing that with few assumptions, systems that are linear in the flux is stable.

The LIFE method evolved from methods in systems biology \cite{palsson}. Systems biology, in conjunction with network flows \cite{heineman, ford}, Markov chains \cite{MarkovChains}, laplacian dynamics \cite{Laplacian}, control theory \cite{piccoli}, and compartmental systems \cite{Bullo,compartmental} allow us to better understand biological networks on which pharmacology models are based. The field of compartmental systems focuses on models based on directed graphs.  Vertices of the graph represent quantities whose dynamics are determined by the edges of the graph, which represent fluxes among compartments.  Markov chains study dynamics on directed graphs as well, but by contrast, this field focuses on stochastic processes. Control theory studies the way an external agent can alter the natural evolution of a system, given a set of admissible controls.  In pharmacology, metabolism follows its natural evolution, and drugs serve as our controls.  These fields have much to contribute to systems pharmacology, and we summarize useful results. We identify assumptions which are usually satisfied by real metabolic networks, guaranteeing stability of the metabolic system at a unique equilibrium.

The paper is organized as follows. In section \ref{sec:sysmodel},
we describe the model system for the LIFE approach in the form of a system of ODEs associated to a metabolic networks, then
show existence of positive solutions and provide results of equilibria under general assumptions.
Also special classes of LIFE systems are introduced.
Section \ref{sec:fix-x} investigates the flow vectors for which a given metabolite vector $x$ is an equilibrium of the network,
including the extreme pathways approach.
On the other side, Section \ref{sec:fix-f} studies the opposite problem: find the metabolite vectors which are
equilibria of the network for a fixed flow vector. This is done first investigating the relationships
between linear LIFE systems and Markov chains, Laplacian dynamics and linear compartmental systems.
Then we deal with special classes of nonlinear LIFE systems.
Finally, a comparison between zero-deficiency theory is discussed.
The paper ends with conclusions in Section \ref{sec:conclusion} and an Appendix containing examples.

\section{System model}\label{sec:sysmodel}
\subsection{Notation and preliminaries}
We indicate by $\R_+ = [0, +\infty)$ the set of positive real numbers,
by $\R^n$ the Euclidean real space of dimension $n$
and by $M_{n\times m}$ the set of $n\times m$ matrices with real entries.
Given a matrix $S$, we indicate by $S^T$ its transpose. Given $d_1,\ldots,d_n\in\R$,
$\diag(d_1,\ldots,d_n)$ is the diagonal matrix with entries $d_i$ on the diagonal.
We denote by $\mathbf 1$ a column vector with all entries equal to $1$, of size clear from the context.

We introduce some terminology commonly used in graph theory.
A directed graph is a couple $G=(V,E)$, with $V=\{v_1,\ldots,v_n\}$ the set of vertices
and $E\subset V\times V$ the set of edges.
For a graph with $n$ vertices and $m$ edges, ordering the edges lexicographically, the \emph{incidence matrix} is a matrix, $\Gamma \in M_{n\times m}$ such that $\Gamma_{ij} = 1$ if the $j$th edge is $(v_k,v_i)$ for some vertex $v_k$, $\Gamma_{ij} = -1$ if the $j$th edge is $(v_i,v_k)$ for some vertex $v_k$, and  $\Gamma_{ij} = 0$ otherwise.
A path is a sequence of distinct vertices $v_{i_1}\cdots v_{i_k}$, with $(v_{i_j},v_{i_{j+1}})\in E$ for $j=1,\ldots,k-1$.
A graph is \emph{strongly connected} if there exists a path between every pair of vertices.
A \emph{strongly connected component} of a directed graph is a maximal strongly connected subgraph.

A \emph{terminal component} of a directed graph $G=(V,E)$ is a strongly connected component $G'=(V',E')$, with $V'\subset V$, $E'\subset E$, such that there exists no edge $e=(v',v)$, with  $v'\in V'$ and $v\in V\setminus V'$.
An undirected path is a sequence of distinct vertices $v_{i_1},\cdots, v_{i_k}$, with either $(v_{i_j},v_{i_{j+1}})\in E$ or
$(v_{i_{j+1}},v_{i_j})\in E$ for $j=1,\ldots,k-1$. A directed graph is \emph{weakly connected}
if there exists an undirected path between every pair of vertices.
A \emph{weakly connected component} of a directed graph is a maximal weakly connected subgraph.
A directed graph $G=(V,E)$ is \emph{weakly reversible} if every weakly connected component is also strongly connected.

\subsection{LIFE model}
We indicate by $x\in\R^n$ the metabolite variables and by $f\in\R^m$ the flux variables.
A general system of ODEs which governs the quantities of $x$ and $f$ is written as
\begin{align}
\frac{\mathrm dx}{\mathrm dt} = F(x,f),\label{dynamics-metabolites}\\
\frac{\mathrm df}{\mathrm dt} = G(x,f)\label{dynamics-fluxes},
\end{align}
where $F:\R^n\times\R^m\to\R^n$ and $G:\R^n\times\R^m\to\R^m$.
In \cite{time-scale,klinke}, the authors show that the dynamics described by \eqref{dynamics-metabolites} evolve over a much smaller time-scale than \eqref{dynamics-fluxes}. This is referred to as ``time-scale separation''. Based on time-scale separation arguments of metabolic systems, we approximate the dynamics of the fluxes with $G \approx 0$, and our work focuses on the dynamics of the metabolites (\ref{dynamics-metabolites}), with the fluxes playing the role of constant  parameters.

The dynamics \eqref{dynamics-metabolites} is very general and we restrict to special system,
which are linear in the fluxes, thus can be written as:
\begin{equation}\label{met-dyn-LIFE}
\dot{x}=S(x)\cdot f
\end{equation}
where $S:\R^n\to M_{n\times m}$ is called the \emph{stoichiometric matrix}.
One constructs the stoichiometric matrix from the metabolites and the reactions that comprise a biochemical system.
Each reaction corresponds to a flux $f$ that connects two distinct metabolites or represents an intake or an excretion from the network.
Each row of $S$ corresponds to a metabolite and each column of $S$ corresponds to a flux.

We now illustrate how to construct a directed graph from the metabolic network
for the system \eqref{met-dyn-LIFE}.
We represent metabolites with vertices $V=\{v_1,\ldots,v_n\}$.
We construct a set of edges $E \subset V \times V$ to represent reactions; each edge is associated to a flux from one metabolite to another, notice that we do not have loops.
To represent intakes and excretions, we introduce two virtual vertices, $v_0$ and $v_{n+1}$, not associated with any metabolite but rather representing the external environment.
We denote by $I,X$ the set of vertices attached to $v_0, v_{n+1}$,
the vertices in $I$ and $X$ are called \emph{intake vertices} and \emph{excretion vertices}, respectively.
We also introduce edges $(v_0,w)$ with $w\in I \subset V$ representing intakes, and $(w,v_{n+1})$, $w\in X \subset V$ representing excretions.
We use the extended graph $\tilde G = (\tilde V, \tilde E)$ defined by $\tilde V = V \cup \{v_0, v_{n+1}\} = \{v_0, v_1, \dots, v_n, v_{n+1}\}$ and $\tilde E$ collecting edges in $E$ together with intake and excretion edges.
The rows of the matrix $S$ can be indexed by vertices in $V$ and the columns by edges in $\tilde E$,
thus we write $S_{ve}$ for the entry corresponding to vertex $v$ and edge $e$.
Moreover we denote by $x_v$ the metabolite corresponding to vertex $v$
and by $f_e$ the flux corresponding to edge $e$.

\begin{example}
To illustrate the concepts of graph with virtual vertices and stoichiometric matrix
related to a metabolic network, we provide a toy example with linear dynamics.
Consider the system given by the following stoichiometric matrix and fluxes vector:
\begin{equation*}
\newcommand\scalemath[2]{\scalebox{#1}{\mbox{\ensuremath{\displaystyle #2}}}}
S(x)=
\scalemath{1}{
\begin{pmatrix}
 1 & -x_1 & -x_1 & 0   & 0    & 0   &  x_4 \\
 0 & x_1  & 0    &-x_2 & 0    & 0   & 0\\
 0 & 0    & x_1  & x_2 & -x_3 & -x_3& 0   \\
 0 & 0    & 0    & 0   &  x_3 &0    & -x_4\\
\end{pmatrix},
} \,
f = \begin{pmatrix}
f_{(v_0,v_1)} \\ f_{(v_1,v_2)} \\ f_{(v_1,v_3)} \\ f_{(v_2,v_3)} \\ f_{(v_3,v_4)} \\
f_{(v_3,v_5)} \\ f_{(v_4,v_1)}
\end{pmatrix}.
\end{equation*}
Then the corresponding graph is represented in Figure \ref{fig:example_network}.
\end{example}

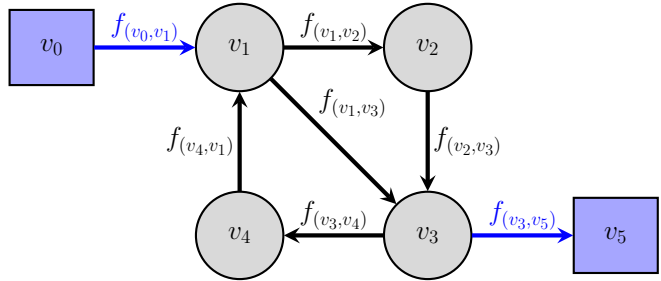
\begin{figure}[h]
\begin{center}
\usetikzlibrary{shapes.geometric, arrows, decorations.markings}

\tikzstyle{metab} = [circle, minimum width=2cm, minimum height=2cm,text centered, draw=black, fill=gray!30]
\tikzstyle{arrow} = [ultra thick,->,>=stealth]
\tikzstyle{vecArrow} = [thin]

\begin{tikzpicture}[node distance=0.6cm,thick,scale=0.5, every node/.style={transform shape}]

\node (source) [metab, rectangle, fill=blue!35, xshift=-5cm, yshift=5cm, text width=2cm] {\Huge $v_0$};
\node (x1) [metab, xshift=0cm, yshift=5cm, text width=2cm] {\Huge $v_1$};
\node (x2) [metab, xshift=5cm, yshift=5cm, text width=2cm] {\Huge $v_2$};
\node (x3) [metab, xshift=5cm,yshift=0cm, text width=2cm] {\Huge $v_3$};
\node (x4) [metab, xshift=0cm,yshift=0cm, text width=2cm] {\Huge $v_4$};
\node (sink) [metab, rectangle, fill=blue!35, xshift=10cm, yshift=0cm, text width=2cm] {\Huge $v_{5}$};

\draw[arrow, blue] (source) -- (x1) node[midway,above,rotate=0] {\Huge $f_{(v_0, v_1)}$};
\draw[arrow] (x1) -- (x2) node[midway,above,rotate=0] {\Huge $f_{(v_1, v_2)}$};
\draw[arrow] (x2) -- (x3) node[midway,right,rotate=0] {\Huge $f_{(v_2, v_3)}$};
\draw[arrow] (x3) -- (x4) node[midway,above,rotate=0] {\Huge $f_{(v_3, v_4)}$};
\draw[arrow] (x4) -- (x1) node[midway,left,rotate=0] {\Huge $f_{(v_4, v_1)}$};
\draw[arrow] (x1) -- (x3) node[midway, above,xshift=0.5cm,yshift=0.5cm] {\Huge $f_{(v_1, v_3)}$};
\draw[arrow, blue] (x3) -- (sink) node[midway,above,rotate=0] {\Huge $f_{(v_3, v_5)}$};
\end{tikzpicture}
\caption{A directed graph $\tilde G = (\tilde{V},\tilde{E})$ representing a biochemical system. The rectangles indicate virtual vertices and the subgraph of circular vertices and edges connecting them is $G = (V,E)$.}
\label{fig:example_network}
\end{center}
\end{figure}

\begin{remark}
It is worth a reminder that our stoichiometric matrix is different from the traditional one defined by \cite{palsson,van}, in which entries are
 stoichiometric coefficients, i.e. do not depend on metabolites.
\end{remark}

To correctly represent the reactions corresponding to fluxes (which take always strictly positive values), we assume:
\begin{itemize}
\item[(A)] For $x\in (\R_+)^n$, it holds
{\footnotesize\begin{equation*}
S_{v e}(x) = \begin{cases}
H_e(x)>0  & \mbox{if}\ e=(w,v),\ w \in V\ \mbox{and}\ x_v>0\ \mbox{or}\ e=(v_0,v),\ v\in I \\
- H_e(x)<0 & \mbox{if}\ e=(v,w),\ w \in V\ \mbox{and}\ x_v>0\ \mbox{or}\ e=(v,v_{n+1}),\ v\in X\ \mbox{and}\ x_v>0 \\
0 & \mbox{otherwise},
\end{cases}
\end{equation*}}where $H_e:\R^n\to \R$ is a positive continuous function.
\end{itemize}
Notice that Assumption (A) implies that, for each $v \in V$,
\begin{equation} \label{eq:mass-cons}
\sum_{v \in V} S_{v e} (x) = \begin{cases}
H_e(x)  & e=(v_0,\bar{v}),\ \bar{v}\in I,\\
-H_e(x) & e=(\bar{v},v_{n+1}),\ \bar{v}\in X,\\
0 & \mbox{otherwise},
\end{cases}
\end{equation}
namely all columns of $S$ have zero sum, except those corresponding to intakes and excretions, which have positive and negative sum, respectively.
Under Assumption (A), the dynamics \eqref{met-dyn-LIFE} can be interpreted as mass conservation law. Indeed, re-writing \eqref{met-dyn-LIFE} entrywise and using (A), we have
\begin{equation} \label{eq:mass-cons-Kirch}
\dot x_v
=  \sum_{e \in \tilde E} S_{ve}(x) f_e
=  \sum_{w: (w,v) \in \tilde E} H_{(w,v)}(x) f_{(w,v)}
  - \sum_{w: (v,w) \in \tilde E} H_{(v,w)}(x) f_{(v,w)} \,,
\end{equation}
which is the mass balance for metabolite $x_v$: its variation is given by the sum of the incoming flows, minus the sum of the outgoing flows.
This is the analogous Kirchhoff's current law for electrical networks, with the  difference that currents are allowed to take negative values as well, while here metabolite variables are non-negative.

The total mass in the system is $m = \sum_{v\in V} x_v$.
From~\eqref{eq:mass-cons-Kirch} we have
\[ \dot m =
\sum_{v \in I} H_{(v_0,v)}(x) f_{(v_0,v)}
- \sum_{v \in X} H_{(v,v_{n+1})}(x) f_{(v,v_{n+1})} \,.
\]
Clearly, in the case without intakes nor excretions, $\dot m = 0$, i.e. the total mass of a closed system is constant in time.

Another remark which is useful is that, under Assumption~(A), $S(x)$ can be re-written as $S(x) = \Gamma D(x)$, where $D(x)$ is a diagonal matrix of size $m \times m$, with diagonal entries given by $H_e(x)$'s, and $\Gamma$ is  obtained from the incidence matrix of $\tilde G$ by removing the first and last rows (corrseponding to $v_0$ and $v_{n+1}$).  In the particular case without intakes nor excretions, $\Gamma$ is the incidence matrix of $G$.

In the remainder of this section we study the dynamics \eqref{met-dyn-LIFE} under
the very general Assumption (A), while later in the paper we add other assumptions, restricting our attention to systems for which stronger statements can be obtained.
A first important general property is that positivity of solution is guaranteed:
\begin{proposition}\label{prop:positivity}
Consider a system \eqref{met-dyn-LIFE} satisfying (A) and the Cauchy problem with initial datum $x_v(0)=x^v_0$.
Assume that $S$ is locally Lipschitz.
If $f_e>0$ for every $e\in E$ and $x^v_0\geq 0$ for every $v\in V$, then
there exists a local solution $x_v(\cdot)$ defined on $[0,T]$, $T>0$,
and $x_v(t)\geq 0$ for every $t\in [0,T]$.
\end{proposition}
\begin{proof}
Existence follows from Lipschitz condition, while positivity of solution follows from the invariance of the set $\{x:x_v\geq 0\}$.
\end{proof}

In the next Proposition we show that existence of nontrivial equilibria implies
some structure on the network: every vertex $v$ for which there is a directed path from some $w\in I$ to $v$, must also have a directed path from $v$ to some $y\in X$.  This result refines the space of networks with which we are concerned.
More precisely:
\begin{proposition}\label{prop:A_neccessary}
Consider a system \eqref{met-dyn-LIFE} satisfying (A).
Assume there exists an equilibrium $\bar{x}\in (\R_+)^n$
for a flux vector $f$ such that $f_e>0$ for every $e\in\tilde  E$.
Then for every vertex $v \in V$ for which there exists a path from $I$ to $v$,
there exists a path from $v$ to $X$.
\end{proposition}
\begin{proof}
Assume there exists an equilibrium $\bar{x}$  as in the statement
and, by contradiction, a vertex $v$
for which there exists a path from $w \in I$ to $v$,
but there exists no $y\in X$ to which $v$ is connected.
Since there is no path from $v$ to $X$, either $v$ belongs to a terminal component with no excretion, or there is a path from $v$ to a terminal component with no excretion.
Denote by $G_T = (V_T, E_T)$ such a terminal component.
Since there are a path from $w \in I$ to $v$ and a (possibly trivial) path from $v$ to $V_T$, then there  is also a path from $v_0$ to $V_T$. Denote by $v_0, v_1 = w, \dots, v_{\ell-1},  v_{\ell} \in V_T$ one such a path, such that $v_{\ell-1} \notin V_T$ (possibly the path is a single edge, in case $w \in V_T$).

It is easy to show that $x_{v_i} >0$ for all $i=1, \dots, \ell$, as follows.
Considering $e=(v_0, v_1)$, by (A) we have $S_{v_0, e}(\bar{x})=H_e(\bar{x})>0$.
Similarly, for every $e'=(v_1,w')$,  $x_{v_1}=0$ implies $S_{v_1,e'}(x)=0$. This means we must have $\bar{x}_{v_1}>0$, otherwise we would have
$\dot{x}_{v_1} \ge f_e H_{e}(\bar{x})>0$ (where $e=(v_0, v_1)$), contradicting $\bar{x}$ being an equilibrium.
Having proved that $\bar x_{v_1}>0$, we can proceed by induction: for $i=1, \dots, \ell-1$,  $\bar x_{v_i}>0$ implies $\bar x_{v_{i+1}}>0$. The argument is the same as above, with a slight modification: looking at $e=(v_{i}, v_{i+1})$, $S_{v_i, e}(\bar x) = H_e(\bar{x})>0$
thanks to Assumption (A) together with $\bar x_i>0$, while above we were in the case of an intake.

Finally we have a terminal component $G_T =(V_T, E_T)$ with no excretion, and an edge $(v_{\ell-1}, v_\ell)$ with $v_\ell \in V_T$ and $v_{\ell-1}\notin V_T$, such that either $v_{\ell-1} = v_0$ or $\bar x_{v_{\ell-1}}>0$. In either case, considering $e=(v_{\ell-1}, v_\ell)$, by (A) we have $S_{v_{\ell-1},e}(\bar x) = H_e(\bar x)>0$.
Now consider the variation of mass in the component $G_T$:
since there is no outgoing edge from $G_T$, and there is at least the incoming edge $e$, we have
$ \frac{\mathrm d}{\mathrm d t} \sum_{v\in V_T} x_v
=  \sum_{v\in V_T} \dot x_v
\ge  H_{e} f_e >0$, contradicting the fact that $\bar{x}$
is an equilibrium.
\end{proof}
\begin{figure}[h]
\begin{center}
\usetikzlibrary{shapes.geometric, arrows, decorations.markings}

\tikzstyle{metab} = [circle, minimum width=2cm, minimum height=2cm,text centered, draw=black, fill=gray!30]
\tikzstyle{arrow} = [ultra thick,->,>=stealth]
\tikzstyle{vecArrow} = [thin]

\begin{tikzpicture}[node distance=0.6cm,thick,scale=0.5, every node/.style={transform shape}]

\node (source) [metab, rectangle, fill=blue!35, xshift=-5cm, yshift=5cm, text width=2cm] {\Huge $v_0$};
\node (x1) [metab, xshift=0cm, yshift=5cm, text width=2cm] {\Huge $v_1$};
\node (x2) [metab, xshift=5cm, yshift=5cm, text width=2cm] {\Huge $v_2$};
\node (x3) [metab, xshift=5cm,yshift=0cm, text width=2cm] {\Huge $v_3$};
\node (x4) [metab, xshift=0cm,yshift=0cm, text width=2cm] {\Huge $v_4$};
\node (sink) [metab, rectangle, fill=blue!35, xshift=10cm, yshift=0cm, text width=2cm] {\Huge $v_{5}$};

\draw[arrow, blue] (source) -- (x1) node[midway,above,rotate=0] {\Huge $f_{(v_0, v_1)}$};
\draw[arrow] (x1) -- (x2) node[midway,above,rotate=0] {\Huge $f_{(v_1, v_2)}$};
\draw[arrow, blue] (x2) -- (sink) node[midway,right,xshift=0.5cm, yshift=0.5cm] {\Huge $f_{(v_2, v_5)}$};
\draw[arrow] (x3) to[bend right] (x4);
\draw[arrow] (x4) to[bend right] (x3);
\draw[arrow] (x2) -- (x4) node[midway,sloped, above] {\Huge $f_{(v_2, v_4)}$};
\node(label1) [draw=none, xshift=2.6cm, yshift=0.5cm] {\Huge $f_{(v_3, v_4)}$};
\node(label2) [draw=none, xshift=2.6cm, yshift=-.5cm] {\Huge $f_{(v_4, v_3)}$};
\end{tikzpicture}
\caption{A directed graph $\tilde G = (\tilde{V},\tilde{E})$ illustrating Proposition \ref{prop:A_neccessary}. Vertices $v_3$ and $v_4$ form a terminal component.  There exists a path from $v_0$ to $v_4$ yet there is no path from $v_4$ to $v_5$}\label{fig:example_network}
\end{center}\vspace*{-10pt}
\end{figure}
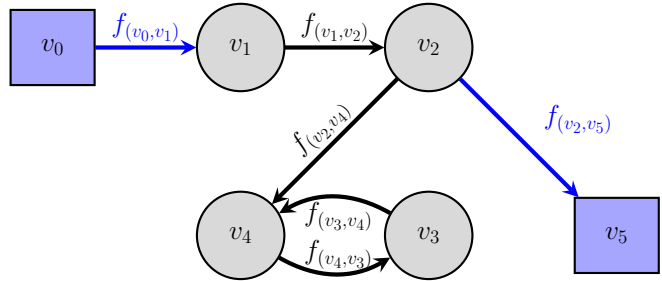
The system associated to the graph in Figure \ref{fig:example_network} provides an explicit example
of the contradiction argument of Proposition \ref{prop:A_neccessary}.
The vertices $v_3$ and $v_4$ violate the condition of Proposition \ref{prop:A_neccessary}, since there is a path from the intake vertex $v_1$ to $v_3$ and $v_4$, and there is no path from $v_3$ and $v_4$ to the excretion vertex $v_2$; hence, the system does not admit any equilibrium.
The proof argument, specialized to this example, is to notice that $v_3$ and $v_4$ form a terminal component with no excretion, and to look at the path $v_0, v_1, v_2, v_4$. Assuming that there is an equilibrium $\bar x$, one shows first that $\bar x_{v_1}>0$ due to the intake, and then that $\bar x_{v_1}>0$ implies also $\bar x_{v_2}>0$.
Notice that  $\dot x_{v_3} + \dot x_{v_4} \ge  S_{v_2,(v_2,v_4)} f_{v_2, v_4} $ (in this particular example, equality is actually true), and the latter is $>0$ since $\bar x_{v_2}>0$,
thus contradicting the fact that $\bar x$ is an equilibrium: the mass $x_{v_3} + x_{v_4}$ grows unbounded.

The same system of Figure \ref{fig:example_network} shows the necessity of the assumption that
$H_e(x)=0$ when $x_v=0$, $e=(v,w)$ for
Proposition \ref{prop:A_neccessary} to hold.
Indeed  assume that $H_{(v_2,v_5)}(x)\break= 1$ for every $x$ such that $x_{v_2}=0$
and that the other functions $H_e$, $e=(v,w)$, satisfy $H_e(x)=x_v$.
Let $f$ be a flux vector with strictly positive components such that
$f_{(v_1,v_2)}=f_{(v_2,v_5)}$.  Then $\bar{x}=(\frac{f_{(v_0,v_1)}}{f_{(v_1,v_2)}},
0, \frac{f_{(v_4,v_3)}}{f_{(v_3,v_4)}}, 1 )$ is an equilibrium.

Another example not satisfying the assumptions of Proposition \ref{prop:A_neccessary} is shown in Figure \ref{fig:example3}. This system admits an equilibrium under Assumption (A) even though the vertex $v_3$ does not have a path to $v_5$. This is because there does not exist a path from $v_0$ to $v_3$.
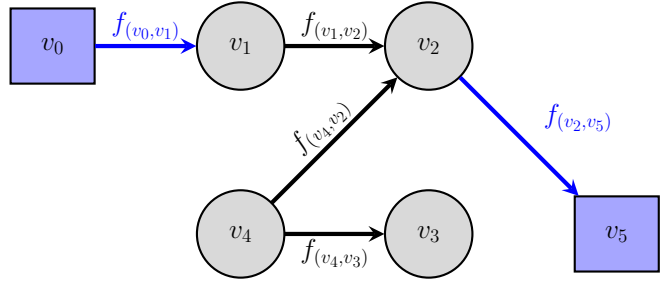
\begin{figure}[h]
\begin{center}
\tikzstyle{metab} = [circle, minimum width=2cm, minimum height=2cm,text centered, draw=black, fill=gray!30]
\tikzstyle{arrow} = [ultra thick,->,>=stealth]
\tikzstyle{vecArrow} = [thin]

\begin{tikzpicture}[node distance=0.6cm,thick,scale=0.5, every node/.style={transform shape}]

\node (source) [metab, rectangle, fill=blue!35, xshift=-5cm, yshift=5cm, text width=2cm] {\Huge $v_0$};
\node (x1) [metab, xshift=0cm, yshift=5cm, text width=2cm] {\Huge $v_1$};
\node (x2) [metab, xshift=5cm, yshift=5cm, text width=2cm] {\Huge $v_2$};
\node (x3) [metab, xshift=5cm,yshift=0cm, text width=2cm] {\Huge $v_3$};
\node (x4) [metab, xshift=0cm,yshift=0cm, text width=2cm] {\Huge $v_4$};
\node (sink) [metab, rectangle, fill=blue!35, xshift=10cm, yshift=0cm, text width=2cm] {\Huge $v_{5}$};

\draw[arrow, blue] (source) -- (x1) node[midway,above,rotate=0] {\Huge $f_{(v_0, v_1)}$};
\draw[arrow] (x1) -- (x2) node[midway,above,rotate=0] {\Huge $f_{(v_1, v_2)}$};
\draw[arrow, blue] (x2) -- (sink) node[midway,right,xshift=0.5cm, yshift=0.5cm] {\Huge $f_{(v_2, v_5)}$};
\draw[arrow] (x4) -- (x3) node[midway,below,rotate=0] {\Huge $f_{(v_4, v_3)}$};
\draw[arrow] (x4) -- (x2) node[midway,sloped, above] {\Huge $f_{(v_4, v_2)}$};
\end{tikzpicture}
\caption{A directed graph where vertices $v_3$ and $v_4$ do not have a path from $v_0$ and also have no path to $v_5$.  For an equilibrium, $\bar{x}$, of this system under Assumption (A), $\bar{x}_{v_4} = 0$ and $\bar{x}_{v_3} \geq x_{v_3}(0)$.}\label{fig:example3}
\end{center}\vspace*{-10pt}
\end{figure}

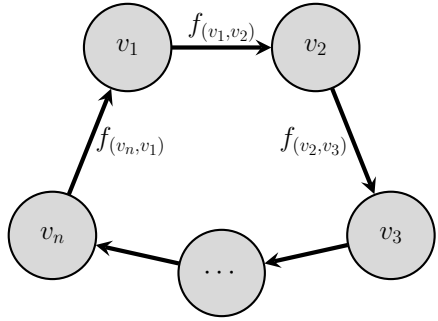
\begin{figure}[h]
\begin{center}
\usetikzlibrary{shapes.geometric, arrows, decorations.markings}

\tikzstyle{metab} = [circle, minimum width=2cm, minimum height=2cm,text centered, draw=black, fill=gray!30]
\tikzstyle{arrow} = [ultra thick,->,>=stealth]
\tikzstyle{vecArrow} = [thin]

\begin{tikzpicture}[node distance=0.6cm,thick,scale=0.5, every node/.style={transform shape}]

\node (x1) [metab, xshift=0cm, yshift=5cm, text width=2cm] {\Huge $v_1$};
\node (x2) [metab, xshift=5cm, yshift=5cm, text width=2cm] {\Huge $v_2$};
\node (x3) [metab, xshift=7cm,yshift=0cm, text width=2cm] {\Huge $v_3$};
\node (xn) [metab, xshift=-2cm,yshift=0cm, text width=2cm] {\Huge $v_n$};
\node (x4) [metab, xshift=2.5cm,yshift=-1cm, text width=2cm] {\Huge $\cdots$};

\draw[arrow] (x1) -- (x2) node[midway,above,rotate=0] {\Huge $f_{(v_1, v_2)}$};
\draw[arrow] (x2) -- (x3) node[midway,left,rotate=0] {\Huge $f_{(v_2, v_3)}$};
\draw[arrow] (x3) -- (x4);
\draw[arrow] (x4) -- (xn);
\draw[arrow] (xn) -- (x1) node[midway,right,rotate=0] {\Huge $f_{(v_n, v_1)}$};

\end{tikzpicture}
\caption{A directed cycle graph $G = (V,E)$ with $n$ vertices and no intakes nor excretions.  On such a LIFE system one can prescribe any desired dynamics.\label{fig:example2}}
\end{center}
\end{figure}

For general systems \eqref{met-dyn-LIFE} with Assumption (A) it is not possible to prove
other general conclusions about equilibria, beside Propositions \ref{prop:positivity} and \ref{prop:A_neccessary}.
Indeed consider the simple metabolic network given in Figure \ref{fig:example2}.
Under Assumption (A), the dynamics is written as:
\begin{equation}\label{eq:cycle_system}
\begin{cases}
\dot{x}_{v_1} &= -f_{({v_1},{v_2})}\cdot H_{(v_1,v_2)}(x) + f_{({v_n},{v_1})}\cdot H_{(v_n,v_1)}(x)\\
\dot{x}_{v_2} &= -f_{({v_2},{v_3})}\cdot H_{(v_2,v_3)}(x)+f_{({v_1},{v_2})}\cdot H_{(v_1,v_2)}(x)\\
\dot{x}_{v_3} &= -f_{({v_3},{v_4})}\cdot H_{(v_3,v_4)}(x)+f_{({v_2},{v_3})}\cdot H_{(v_2,v_3)}(x)\\
& \vdots \\
\dot{x}_{v_n} &= -f_{({v_n},{v_1})}\cdot H_{(v_n,v_1)}(x)+f_{({v_{n-1}},{v_n})}\cdot H_{(v_{n-1},v_n)}(x)
\end{cases}.
\end{equation}
We want to show that for any dynamical system in $\R^n$ on a compact set
there exists an equivalent dynamics defined on the cycle graph of Figure \ref{fig:example2}.  In other words, Asumption (A) is so general that we give Proposition \ref{prop:gen-dyn-Rn} to show arbitrary dynamics can be defined, and we focus on more specialized cases (Assumption (B), and (C)).
More precisely for every general dynamics
\begin{equation}\label{eq:sys-Rn}
\begin{cases}
\dot{x}_{v_1}&=F_1(x_{v_1},x_{v_2},\ldots, x_{v_n})\\
\dot{x}_{v_2}&= F_2(x_{v_1},x_{v_2},\ldots, x_{v_n})\\
&\vdots \\
\dot{x}_{v_n} &= F_n(x_{v_1},x_{v_2},\ldots, x_{v_n})\\
\end{cases}.
\end{equation}
we look for a choice of the functions $H_e$ and the fluxes $f_e$ realizing such equivalence.
Note that from \eqref{eq:mass-cons}, we have $x_{v_n} = C - x_1 - \cdots - x_{n-1}$ which implies that
\begin{equation}\label{eq:xvn}
\dot{x}_{v_n} = -\dot{x}_{v_1}-\cdots - \dot{x}_{v_{n-1}} = \sum_{i=1}^{n-1} - F_i.
\end{equation}
Define the set $T = \cup_{i=1}^n \{x:  x_j \geq 0 \text{ for all } j,\ x_i=0 \text{ and } x_1+\cdots+x_n \leq 1 \}$ then we have the following:
\begin{proposition}\label{prop:gen-dyn-Rn}
Consider system \eqref{eq:sys-Rn} and assume $F_1=\cdots=F_n=0$ on the set $T$.
Then there exist functions $H_e$ and fluxes $f_e$ such that the dynamics \eqref{eq:cycle_system}
is equivalent to \eqref{eq:sys-Rn} on the bounded set delimited by $T$.
\end{proposition}
\begin{proof}
Assign the functions $H_{(v_i,v_j)}(x)$ according to the following rule:
\begin{equation}
\begin{split}
H_{(v_i,v_{i+1})}(x) &=  \sum_{k\leq i}[F_k]_- + \sum_{\ell> i}[F_\ell]_+  \quad \text{if }\ i<n,\\
H_{(v_n,v_1)}(x) &= \sum_{k=1}^n [F_k]_+,
\end{split}
\end{equation}
where $[F]_+=\max\{F,0\}$ and $[F]_-=\max\{-F,0\}$. It is easy to verify that $\dot{x}_{v_i}= F_i(x)$ for $i<n$
and $\dot{x}_{v_n}= - \sum_{i=1}^n F_i=F_n(x)$ using  \eqref{eq:xvn}.
Moreover, because of the assumption on $F_i$'s,
the functions $H_e$ satisfy (A), thus we are done.
\end{proof}

\subsection{Special LIFE systems}

Here we introduce a special class of systems of type \eqref{met-dyn-LIFE}
with simplified dynamics.
We consider Assumption (A), and we impose further restriction on the functions $H_e(x)$: for an edge $e = (v,w)$, we assume $H_e$ to depend on $x_v$ only, and moreover we impose the scalar function $H_e(x_v)$ to be strictly increasing. More precisely,
Assumption (B) is the following.
\begin{itemize}
\item[(B)] It holds
\begin{equation} \label{eq:S_construction}
S_{v e}(x) = \begin{cases}
 -H_e(x_v)  & e=(v,w),\ v \in V, w \in V\cup \{v_{n+1}\}\\
 H_e(x_w) & e=(w,v),\ w \in V \\
 1 & e=(v_0,v)\ v\in I \\
0 & \mbox{otherwise},
\end{cases}
\end{equation}
\noindent
and each $H_e$ is a positive, differentiable, and strictly increasing function $H_e:\R \to \R,$ with $H_e(0)=0$.
\end{itemize}

A typical example of a system verifying (B) is given by metabolic networks
with Hill functions representing reactions, i.e. $H_e(x_v)= \frac{x_v^{p_e}}{K + x_v^{p_e}}$ with $p_e \in \mathbb{N}$, where $K>0$ is the dissociation constant.

A further simplification occurs in the case where $H_e(x_v)$ is the same function $H_v(x_v)$ for all edges $e$ having $v$ as an initial vertex, i.e. such that $e=(v,w)$ for some $w$.
This gives Assumption (C), as follows.
\begin{itemize}
\item[(C)] It holds
\begin{equation*} 
S_{v e}(x) = \begin{cases}
 -H_v(x_v)  & e=(v,w),\ v \in V, w \in V\cup \{v_{n+1}\}\\
 H_w(x_w) & e=(w,v),\ w \in V \\
 1 & e=(v_0,v)\ v\in I \\
0 & \mbox{otherwise},
\end{cases}
\end{equation*}
and each $H_v$ is a positive, differentiable, and strictly increasing function $H_v:\R \to \R,$ with $H_v(0)=0$.
\end{itemize}

Under Assumption~(C), the system $\dot x = S(x) f$ can be equivalently re-written as
\begin{equation} \label{eq:dyn-J}
\dot x = J(f) h(x) + \phi \,,
\end{equation}
where $J(f) \in M_{n\times n}$ is defined by
\[
J_{ij}(f)= \begin{cases}
f_{(v_j,v_i)} & \text{if $(v_j,v_i) \in E$} \\
- \sum_{w: (v_i, w) \in \tilde E} f_{(v_i,w)} & \text{if} j=i \\
0 & \text{otherwise},
\end{cases}
\]
where $h(x)$ is a vector of size $n$ given by
$h_i(x) = H_{v_i}(x_{v_i})$ and $\phi$ is a vector of size $n$ given by $\phi_{i} = f_{(v_0,v_i)}$ if $(v_0,v_i) \in \tilde E$, and $\phi_i = 0 $ otherwise.

Finally, the simplest class of LIFE models we consider are \emph{linear} systems, namely systems satisfying Assumption~(C), where each $H_v(x_v)$ is the identity function $H_v(x_v) = x_v$.

\vspace*{4pt}\noindent\textbf{Example 2.1 (continued).}
This example is a linear LIFE system. We can equivalently re-write its dynamics $\dot x = S(x) f$ as $\dot x = J(f) x + \phi$, with
{\small\[ J(f) =
\left( \begin{matrix}
-\! f_{(v_1,v_2)} \!-\! f_{(v_1,v_3)} & 0& 0& f_{(v_4,v_1)}\\
f_{(v_1,v_2)} & -\! f_{(v_2,v_3)}&0&0\\
f_{(v_1,v_3)} & f_{(v_2,v_3)} & -\! f_{(v_3,v_4)} \!-\! f_{(v_3,v_5)} &0\\
0 & 0 & f_{(v_3,v_4)} & -\!f_{(v_4,v_1)}\\
\end{matrix} \right), \;
\phi =
\left(\begin{matrix}
f_{(v_0,v_1)} \\ 0 \\ 0 \\ 0
\end{matrix} \right).
\]}
\section{Equilibria for fixed metabolites} \label{sec:fix-x}
\newcommand\scalemath[2]{\scalebox{#1}{\mbox{\ensuremath{\displaystyle #2}}}}

In this section we consider equilibrium solutions of the system \eqref{met-dyn-LIFE} satisfying Assumption (A).
In general one is interested in conditions guaranteeing existence of an equilibrium and also in conditions necessary
for uniqueness and stability of such an equilibrium.
There are two problems: for fixed metabolite concentrations $x$ find all flux vectors $f$ for which $x$ is an equilibrium,
i.e.  $\dot{x}=S(x)\cdot f = 0$,
and, vice versa, for a fixed flux vector $f$, find all $x$ that are equilibria.
In this section we focus on the first, while the latter is investigated in Section~\ref{sec:fix-f}.

The set of flux vectors for which $x$ is an equilibrium formed by all vectors $f$ that solve the equation $S(x)\cdot f = 0$, i.e.
the nullspace $\mathcal{N}(S(x))$ of $S(x)$.
First, we discuss the dimension of $\mathcal{N}(S(x))$. Then, since fluxes must be positive (to have a correct biological meaning),
we focus on describing the cone $\mathcal{N}(S(x)) \cap (\mathbb{R}_+)^{m}$.
Recall that, under Assumption~(A), we can rewrite $S(x)$ as  $S(x) = \Gamma D(x)$ where $D(x) \in M_{m\times m}$ is a diagonal matrix with
$H_e(x)$'s as entries and $\Gamma$ is obtained from the incidence matrix of $\tilde G$ by removing the first and last rows; in case there are no intakes nor excretions, then $\Gamma$ is the incidence matrix of $G$.
Assuming that $x$ has strictly positive entries, $(A)$ implies that all diagonal elements of $D(x)$ are strictly positive, so that $D(x)$ is invertible. Hence,
for $\Gamma \cdot (D(x) \cdot f) = 0$ to have non trivial solutions, the nullspace of $\Gamma$ must have dimension greater than zero.

We extend one of the results of \cite{ACC} to get:
\begin{proposition}\label{prop:rankS}
Consider the system \eqref{met-dyn-LIFE} with no intakes nor excretions satisfying Assumption (A) and let $G$ be the the associated graph. Let $x \in \R^n$ have strictly positive entries and $S(x) \in M_{n\times m}$ be the stoichiometric matrix, then
$$\rank(S(x)) = n-\ell,$$ where $\ell$ is the number of weakly connected components of $G$.
\end{proposition}
\begin{proof}
We have $S(x) = \Gamma D(x)$. As shown in Proposition 4.3 of \cite{Biggs} the rank of an incidence matrix $\Gamma$ is $\rank(\Gamma) =n-\ell$.  Since $D(x)$ is full rank we have $\rank(S(x)) = n-\ell$.
\end{proof}

Next we consider systems that contain both intakes and excretions. It was shown in \cite{ACC} that if intakes and excretions are added to a graph satisfying the conditions of Proposition \ref{prop:rankS} then $\rank(S(x)) = n$. We extend this result to get:

\begin{proposition}\label{prop:rankS2}
Consider the system \eqref{met-dyn-LIFE} satisfying Assumption (A) and let $G$ be the the associated graph. Let $x \in \R^n$ have strictly positive entries and $S(x) \in M_{n\times m}$ be the stoichiometric matrix, then
$\rank(S(x)) = n - k$, where $k$ is the number of weakly connected components containing neither intake nor excretion vertices.
\end{proposition}
\begin{proof}

It was shown in \cite{ACC} that if $G$ is weakly connected and contains an intake vertex then $\rank(S(x)) = n$. The same argument from \cite{ACC} can be also used for an excretion vertex and so if $G$ is weakly connected and contains an excretion vertex we also have $\rank(S(x)) = n$. In the case where $G$ is not weakly connected $S(x)$ can be rewritten as a block diagonal matrix $$S(x) = \begin{pmatrix}
S_1(x) & & 0\\
& \ddots & \\
0& & S_\ell(x)
\end{pmatrix}$$ where each diagonal element $S_i(x)$ is the stoichiometric matrix of a weakly connected component. If $S_i(x)$ contains an intake or excretion vertex then $rank(S_i(x))\break = n_i$, where $n_i$ is the number of metabolites in $S_i(x)$. If $S_i(x)$ contains neither intake nor excretion vertices then $rank(S_i(x)) = n_i - 1$. This implies that $S(x)$ has $n_i$ linearly independent rows for each $S_i(x)$ with intake or excretion vertices and $n_i - 1$ linearly independent rows for each $S_i(x)$ with neither intake nor excretion vertices. Let $k$ represents the number of $S_i(x)$ with neither intake nor excretion vertices, the total number of linearly independent rows in $S$ is now expressed $(\sum_{i=1}^\ell (n_i)) - k  = n - k$. Thus $\rank(S(x)) = n-k$.
\end{proof}

Notice that the rank of $S(x)$ depends on the number of weakly connected components of the graph, which is the same irrespective of the orientation of edges. However, when we focus on the existence of non-trivial positive flows admitting an equilibrium with strictly positive entries, the orientation of edges does matter, as we can see e.g.\ from Proposition~\ref{prop:A_neccessary}, where a necessary condition is given for existence of equilibria, in terms of existence of suitable paths.

\subsection{Network flows}

The problem of finding positive flows admitting an equilibrium with strictly positive entries has been extensively studied in the operations research literature under the name of `network flow' problems \cite{heineman}. In network flow problems one considers a directed graph where edges represent flows between the vertices. The maximum flow which an edge can support is called the capacity of the edge. In addition to capacity each edge may also have a cost associated to it. Network flow problems assume that the system is at equilibrium with respect to the vertices i.e. that the flow entering and leaving a vertex must be the same. A flow that satisfies this assumption is called a feasible flow. Finding feasible flows is the same as finding flows in the cone $\mathcal{N}(S(x)) \cap (\mathbb{R}_+)^{m}$ with the further constraint that each flux must be less than its capacity.

A flow is a mapping from $f:E \to (\mathbb{R}_+)^m$ that satisfies $0 \leq f(v_i,v_j) \leq c(v_i,v_j)$ (where $c(v_i,v_j)$ is the capacity of the edge $(v_i,v_j)$) and $\sum_{(v_i,v_j) \in E} f(v_i,v_j) - \sum_{(v_j,v_k) \in E} f(v_j,v_k) = 0$ \cite{malhotra}.
One of the most common network flow problems is to find the maximum flow of the network, i.e. the largest amount of total flow from a source to a sink. In \cite{ford}
the authors consider a network which contains exactly one source ($v_0$) and one sink ($v_{n+1}$). They then prove the following result, known as max-flow min-cut theorem, which characterizes the maximum flow as the minimum cut cacapity, where a cut set is a set of edges whose removal disconnects the source from the sink, and the capacity of a cut is the sum of the capacities of its edges.

\begin{proposition}[Max-Flow Min-Cut Theorem]
The maximum flow value obtainable in a network is the minimum capacity of all cut sets which disconnect $v_0$ and $v_{n+1}$.
\end{proposition}

An implication of the max-flow min-cut theorem is that a feasible flow exists if there is a path from $v_0$ and $v_{n+1}$. The following proposition extends this relationship:

\begin{proposition}\label{prop:feasibleflows}
Given a LIFE system with graph $\tilde G$ and stoichiometric matrix $S(x)$ satisfying (A), fix an $x \in (\R_+)^n$ with strictly positive entries and for every $v \in I$ fix $\bar f_{v_0,v}>0$, i.e. fix arbitrary values for the intake flows.
There exists  $f \in \mathcal{N}(S(x)) \cap (\mathbb{R}_+)^{m}$
such that $f_{v_0,v} = \bar f_{v_0,v}$ for all $v \in I$ if and only if for each $v \in I$ there exists a path to $X$.
\end{proposition}
\begin{proof}
Consider the maximum flow problem on the graph $\tilde G$, where edges $(v_0,v)$ have capacity $\bar f_{v_0,v}$, and all other edges have infinite capacity.
The feasible flows $\varphi$ for this network are in one-to-one correspondence with the equilibrium flows $f \in \mathcal{N}(S(x)) \cap (\mathbb{R}_+)^{m}$ such that $f_{v_0,v} \le \bar f_{v_0,v}$ for all $v \in I$; the correspondence is simply given by $f_{v_0,v} = \varphi_{v_0,v} / H_{v_0,v}(x)$ for all $v\in I$ and $f_{v,w} = \varphi_{v,w} / H_{v,w}(x)$ for all $w \in V$.
If for all $v \in I$ there is a path from $v$ to $X$ (and hence a path from $v$ to $v_{n+1}$),
the minimum cut is the removal of all edges $(v_0,v)$.
The maximum flow $\varphi^*$ then has $\varphi^*_{v_0, v} = \bar f_{v_0,v}$, thus also ensuring the existence of an equilibrium flow $f^*$ satisfying the same.

If for some $v_i \in I$ there is no path from $v$ to $X$, then all feasible flows $\varphi$ satisfy $\varphi_{v_0, v}  = 0$, and hence all equilibrium flows $f$ satisfy $f_{v_0, v}  = 0$ which contradicts the assumption.
\end{proof}

\subsection{Extreme pathway algorithm for calculating the positive basis}\label{sec:algorithm}

There are many standard methods for computing a basis of the nullspace of a matrix   \cite{Meyer}, and hence to describe $\mathcal{N}(S(x))$.
Since we are interested in positive fluxes, we are rather interested in the cone $\mathcal{N}(S(x)) \cap (\mathbb{R}_+)^{m}$. One method to describe this cone is to look for a positive basis, which
is a minimal set of vectors generating the whole cone via linear combinations with positive scalars, called positive combinations.
The use of positive combinations ensures that all generated vectors belong to the cone.
Minimality is equivalent to ask for the vectors to be positively linearly independent, i.e. no vector can be expressed as a positive combination of the others. Since we allow only positive combinations (and not every linear combination),
typically a positive basis has more vectors than a traditional basis.
In \cite{Schilling} the authors prove that a flux cone has a positive basis, and that the vectors of the positive basis are unique up to multiplication by a positive scalar.

Finding a positive basis for a cone is not as easy as finding a traditional basis. In \cite{palsson}, Palsson describes a method to find a positive basis by using extreme pathways of a metabolic network. 
The extreme pathways are a unique set of positively independent flux vectors that represent the edges of the cone $\mathcal{N}(S(x)) \cap (\mathbb{R}_+)^{m}$ \cite{palsson,Schilling}.

Here we summarize a method to build extreme pathways described by \cite{Schilling}.
We start considering the equation $V_0 \cdot S(x)^T = C_0$,
and the solution given by $V_0=I_{m}$, the $m\times m$ identity matrix, $S(x)$ the stoichiometric matrix
and $C_0=S(x)^T$. At each iteration new matrices
$V_i\in M_{n_i\times m}$,  $C_i\in M_{n_i\times n}$ are defined which satisfy $V_i \cdot S(x)^T = C_i$.

\begin{algorithm}\label{alg:Nate+Palsson}
The algorithm consists of the following steps:

\vspace*{4pt}\noindent\textbf{Step 1.} Select from $C_i$ the first non-zero column (lexicographically), whose corresponding metabolite
is neither an intake vertex nor an excretion vertex, say $j_i$. If no such column exists jump to Step 5.

\vspace*{4pt}\noindent\textbf{Step 2.} Add to $C_i$ all possible new rows, obtained using positive linear combinations of two other rows, which have a zero on the column $j_i$. Then define $C_{i+1}$ by removing the old rows used to generate the new ones.
Notice that $C_{i+1}$ has all zeros in column $j_i$.

\noindent\textbf{Step 3.} Define $V_{i+1}$ by adding to $V_i$ the rows generated by the same combination as those of Step 2
(to keep the validity of the equation $V \cdot S(x)^T = C$) and removing the old rows.

\noindent\textbf{Step 4.} Remove positively linearly dependent rows from $V_{i+1}$. Remove the corresponding rows from $C_{i+1}$.

\noindent\textbf{Step 5.} Repeat the steps 1-4 now considering first excretion vertices and then intake vertices until all columns are $0$.
\end{algorithm}
If the algorithm stops at step $m$, then $C_{m}=0$ (the matrix with all zero entries).
The rows of $V_{m}$ are the vectors of a positive basis.
Notice also that the number of rows $n_i$ may increase or decrease during the various steps.

Originally in \cite{palsson} the extreme pathways method was proposed for stoichiometric matrices not dependent on $x$. We apply it here to the most general case of systems \eqref{met-dyn-LIFE} with Assumption~(A), by fixing a desired metabolite variables vector $x \in (\R_+)^n$ and using the extreme pathways method to characterize all fluxes vectors with positive entries such that the corresponding LIFE system \eqref{met-dyn-LIFE} admits $x$ as an equilibrium.

To illustrate the process we report an example using the Reverse Cholesterol Transport Network (RCT) from \cite{LIFE}. The RCT network is shown in Figure \ref{fig:RCT Network}.

\begin{figure}[h]
\begin{center}
\usetikzlibrary{shapes.geometric, arrows, decorations.markings}

\usetikzlibrary{shapes.geometric, arrows, decorations.markings}
\tikzstyle{automated} = [rectangle, rounded corners, minimum width=3cm, minimum height=1cm,text centered, draw=black,  fill=cyan!60]
\tikzstyle{metab} = [rectangle, rounded corners, minimum width=3cm, minimum height=1cm,text centered, draw=black,  fill=cyan!40]
\tikzstyle{metab1} = [rectangle, rounded corners, minimum width=1cm, minimum height=0.5cm,text centered, draw=black,  fill=cyan!40]
\tikzstyle{io} = [ellipse , rounded corners, minimum width=3cm, minimum height=1cm,text centered, draw=black,  fill=red!40]
\tikzstyle{io2} = [ellipse , rounded corners, minimum width=0cm, minimum height=0cm,text centered, draw=gray!20,  fill=gray!20]
\tikzstyle{human} = [trapezium, trapezium left angle=70, trapezium right angle=110, minimum width=3cm, minimum height=1cm, text centered, draw=black, fill=orange!60]
\tikzstyle{label} = [rectangle, rounded corners, minimum width=3cm, minimum height=1cm,text centered, draw=black, fill=gray!20]
\tikzstyle{label1} = [rectangle, rounded corners, minimum width=20cm, minimum height=12cm, draw=black, fill=gray!10]
\tikzstyle{label2} = [rectangle, rounded corners, minimum width=30cm, minimum height=8cm, draw=black, fill=gray!50]
\tikzstyle{caption} = [rectangle, minimum width=3cm, minimum height=1cm,text centered, draw=black, fill=gray!10]
\tikzstyle{process1} = [rectangle, minimum width=3cm, minimum height=1cm, text centered, draw=black, fill=orange!20]
\tikzstyle{process} = [rectangle, minimum width=3cm, minimum height=1cm, text centered, draw=black, fill=orange!40]
\tikzstyle{decision} = [rectangle, minimum width=3cm, minimum height=1cm, text centered, draw=black, fill=green!20]
\tikzstyle{flow} = [circle, minimum width=0.5cm, text centered, draw=black, fill=green!30]
\tikzstyle{flow1} = [rectangle, minimum width=1cm, text centered, draw=black, fill=green!30]
\tikzstyle{arrow} = [thick, dashed,->,>=stealth]
\tikzstyle{vecArrow} = [thick, decoration={markings,mark=at position
   1 with {\arrow[semithick]{open triangle 60}}},
   double distance=1.4pt, shorten >= 5.5pt,
   preaction = {decorate},
   postaction = {draw,line width=1.4pt, white,shorten >= 4.5pt}]
\tikzstyle{innerWhite} = [semithick, white,line width=1.4pt, shorten >= 4.5pt]

\tikzstyle{automated} = [rectangle, rounded corners, minimum width=3cm, minimum height=1cm,text centered, draw=black,  fill=cyan!60]

\tikzstyle{results} = [ellipse, text centered, draw=black,  fill=cyan!60]

\tikzstyle{human} = [trapezium, trapezium left angle=70, trapezium right angle=110, minimum width=3cm, minimum height=1cm, text centered, draw=black, fill=orange!60]
\tikzstyle{label} = [rectangle, rounded corners, minimum width=3cm, minimum height=1cm,text centered, draw=black, fill=gray!40]

\tikzstyle{caption} = [rectangle, minimum width=3cm, minimum height=1cm,text centered, draw=black, fill=gray!10]

\tikzstyle{process} = [rectangle, minimum width=3cm, minimum height=1cm, text centered, draw=black, fill=orange!30]

\tikzstyle{decision} = [diamond, minimum width=3cm, minimum height=1cm, text centered, draw=black, fill=green!20]

\tikzstyle{answer} = [circle, minimum width=1.5cm, text centered, draw=black, fill=orange!40]

\tikzstyle{arrow} = [ultra thick,->,>=stealth]

\tikzstyle{vecArrow} = [thick, decoration={markings,mark=at position
   1 with {\arrow[semithick]{open triangle 60}}},
   double distance=1.4pt, shorten >= 5.5pt,
   preaction = {decorate},
   postaction = {draw,line width=1.4pt, white,shorten >= 4.5pt}]
\tikzstyle{innerWhite} = [semithick, white,line width=1.4pt, shorten >= 4.5pt]

\begin{tikzpicture}[node distance=2cm,thick,scale=0.5, every node/.style={transform shape}]

\node (label1) [label1, text width=12cm] {\Large};
\node (source) [io2, left of= label1,xshift=-4cm, yshift=-4cm, text width=0cm] { };
\node (source2) [io2, right of= source,xshift=10cm, yshift=0cm, text width=0cm] {};
\node (source1) [io, right of= source,xshift=4cm, yshift=0cm, text width=9.5cm] {\huge $v_0$};

\node (x1) [metab, above of= source, xshift=0cm, yshift=1cm, text width=4cm] {\huge $v_1$};
\node (x2) [metab, above of= source1, xshift=0cm, yshift=1cm, text width=4cm] {\huge $v_2$};
\node (x3) [metab, above of= source2, xshift=0cm,yshift=1cm, text width=4cm] {\huge $v_3$};
\node (x4) [metab, above of= x2, xshift=-2cm,yshift=1cm, text width=6cm] {\huge $v_4$};
\node (x5) [metab, above of= x4, xshift=-4cm,yshift=1cm, text width=4cm] {\huge $v_5$};
\node (x6) [metab, above of= x4, xshift=4cm,yshift=1cm, text width=4cm] {\huge $v_6$};

\node (f1) [flow1, above of= source, xshift=-.8cm, yshift=-.6cm] {\huge $f_{(v_0,v_1)}$};
\node (f2) [flow1, above of= source1, xshift=-.8cm, yshift=-.6cm] {\huge $f_{(v_0,v_2)}$};
\node (f3) [flow1, above of= source2, xshift=.8cm, yshift=-.6cm] {\huge $f_{(v_0,v_3)}$};
\node (f4) [flow, above of= x1, xshift=.8cm, yshift=-.5cm] {\huge $f_{(v_1,v_4)}$};
\node (f5) [flow, above of= x2, xshift=-1.7cm, yshift=-.8cm] {\huge $f_{(v_2,v_4)}$};
\node (f6) [flow, above of= x3, xshift=-3.5cm, yshift=0cm] {\huge $f_{(v_3,v_4)}$};
\node (f7) [flow, above of= x4, xshift=-3cm, yshift=-.6cm] {\huge $f_{(v_4,v_5)}$};
\node (f8) [flow, above of= x4, xshift=3cm, yshift=-.6cm] {\huge $f_{(v_4,v_6)}$};
\node (f9) [flow, right of= x5, xshift=2cm, yshift=-.7cm] {\huge $f_{(v_5,v_6)}$};
\node (f10) [flow, right of= x6, xshift=1.8cm, yshift=-0.8cm] {\Large $f_{(v_6,v_{n+1})}$};

\node (sink) [io, right of= x6, xshift=2cm, yshift=-2cm, text width=3cm] {\huge $v_{n+1}$};

\node (label3) [label, right of= x4, xshift=7cm, yshift=-1cm, text width=4cm] {\huge metabolic network};

\draw [vecArrow] (source) -- (x1);
\draw [vecArrow] (x1) -- (x4);
\draw [vecArrow] (x2) -- (x4);
\draw [vecArrow] (x3) -- (x4);
\draw [vecArrow] (source) -- (x1);
\draw [vecArrow] (source1) -- (x2);
\draw [vecArrow] (source2) -- (x3);
\draw [vecArrow] (x4) -- (x5);
\draw [vecArrow] (x4) -- (x6);
\draw [vecArrow] (x5) -- (x6);
\draw [vecArrow] (x6) -- (sink);

\end{tikzpicture}

\caption{Reverse Cholesterol Transport Network from \cite{LIFE}. This network contains 6 vertices which represent metabolites, 10 edges which represent fluxes and 2 virtual vertices $v_0,v_{n+1}$. There are three intake vertices $v_1,v_2,v_3$ and 1 excretion vertex $v_{6}$.}
\label{fig:RCT Network}
\end{center}
\end{figure}
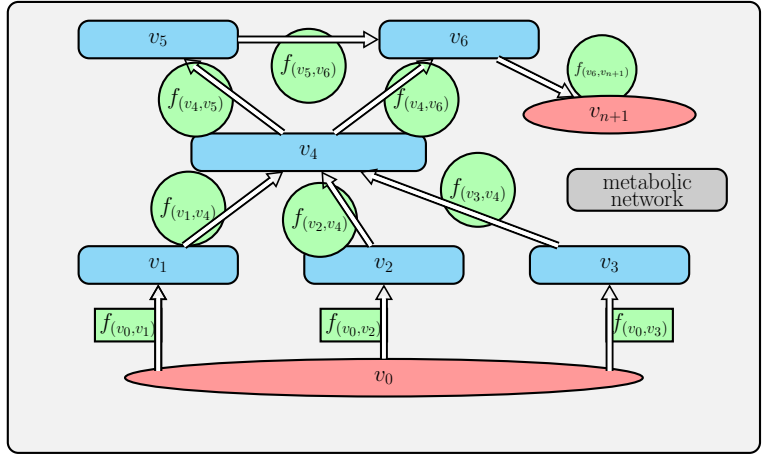

\begin{example}{Finding Extreme Pathways of RCT network}\label{ex:extremes}

We use Palsson's algorithm for finding extreme pathways on the stoichiometric matrix from the RCT network. The stoichiometric matrix for RCT is given by:
\begin{equation*}\label{eq:stoichiometric_matrix}
\scalemath{.93}{
S(x)=\left(\begin{array}{cccccccccc}
 1 & 0 & 0 & -x1 & 0 &0 &0&0&0& 0\\
  0 & 1 & 0 & 0 & -x2 &0 &0&0&0&0\\
 0 & 0 & 1 & 0 & 0 &-x3 &0&0&0&0\\
 0 & 0 & 0 & x1 & x2 &x3 &-x4&-x4&0&0\\
 0 & 0 & 0 & 0 & 0 &0 &x4&0&-x5&0\\
 0 & 0 & 0 & 0 & 0 &0 &0&x4&x5&-x6\\
\end{array}\right)
}
\end{equation*}

For space, the complete calculations are continued in the Appendix.
\end{example}

\subsection{Extreme pathways provide a positive basis}

Let us first recall Farkas' Lemma \cite{Gale}. Here we use the notation $x\geq 0$ to indicate that every entry in the vector $x$ is positive.
\begin{lemma}[Farkas' Lemma]\label{lem:Farka}
Let $A \in M_{n \times m}$ and $b \in R^n$. For the equation $Ax = b$ exactly one is true:\\
1. There exists $x \in \mathbb{R}^m$ such that $Ax=b$ and $x\geq 0$. \\
2. There exists $y \in \mathbb{R}^n$ such that $A^Ty\geq 0$ and $b^T y<0$.\\
\end{lemma}

\begin{proposition}\label{prop:Positive_Basis}
The extreme pathways of $S(x)$ form a positive basis of $\mathcal{N}(S(x)) \cap (\mathbb{R}_+)^{m}$.
\end{proposition}

\begin{proof}
There are three statements which must be shown to prove this Proposition.

\vspace{.2 cm}

\noindent
\textit{Statement 1:}  The extreme pathway vectors ($v_i$) are positively independent
(systematically independent in terminology used in \cite{palsson}).

\noindent
\textit{Statement 2:} The span of the extreme pathways vectors is contained in $(\mathcal{N}(S(x)) \cap (\mathbb{R}_+)^{m})$ i.e. $\sum_{\lambda_i \geq 0} \lambda_iv_i \subset (\mathcal{N}(S(x)) \cap (\mathbb{R}_+)^{m})$.

\noindent
\textit{Statement 3:} The extreme pathways vectors span $(\mathcal{N}(S(x)) \cap (\mathbb{R}_+)^{m})$ i.e. $  (\mathcal{N}(S(x)) \cap (\mathbb{R}_+)^{m}) \subset \sum_{\lambda_i \geq 0} \lambda_iv_i$.

Statement 1 follows from Step 4 of the algorithm.

Let us now prove Statement 2.
Since the vectors corresponding to extreme pathways are positive combinations of rows of matrices with positive entries,
then $v_i \in (\mathbb{R}_+)^{m} \mbox{ for all } i$.
We have $V_n \cdot S(x)^T = C_n=0$, equivalently  $S(x) \cdot V_n^T = C_n^T=0$. This implies that the columns
of $V_n^T$  in the nullspace of $S(x)$, thus the rows of $V_n$ in the nullspace of $S(x)$.
Since both $\mathcal{N}(S(x)$ and $(\mathbb{R}_+)^{m}$ are closed for positive linear combinations we are done.

Finally it remains to prove statement 3, in order to show $(\mathcal{N}(S(x)) \cap (\mathbb{R}_+)^{m}) \subset \sum_{\lambda_i \geq 0} \lambda_iv_i$
we proceed as follows.
First we show that after one iteration of the algorithm
the rows of $V_1$ form a positive basis for vectors $w$ such that
the $j_1$-th element of $S(x) w$ equals zero.
After $n$ iterations, we show that $V_n$ is a positive basis for vectors $w\in \mathcal{N}(S(x))$.

For simplicity of notation, we assume $j_1=1$.
Define $W_1$ to be the set of vectors $w$ such that  the first entry of $S(x) w$ equals zero.
It is clear from the construction of $V_1$ that its rows belong to $W_1$. It remains to be shown that the rows of $V_1$ span $W_1$.
Since $V_1\cdot S(x)^T=C_1$, thus we can alternatively show that
the columns of $C_1^T$ span the subset of the range of $S(x)$ with first entry equal to zero.
In other words, we must show that if $\sum_{\lambda_i \geq 0} \lambda_ic_{i,1} = 0 $, where $c_{i,1}$ represents the first entry in the $c_i$-th row of $C_0$, that the sum
$\sum_{\lambda_i \geq 0} \lambda_ic_{i} = 0 $
can also be represented using rows from $C_1$.
We assume an ordering such that
$c_{i,1} \geq 0$
for $i\leq \ell$
and $c_{i,1}< 0$ for $i> \ell$,
each row in $C_1$ can be represented by
$c_i + \alpha_{i,j}c_j$
where $i\leq \ell, j>\ell$ and $\alpha_{i,j}= \frac{-c_{i,1}}{c_{j,1}} \geq 0$.
We need to find $\mu_{i,j}$ such that,
\begin{equation}\label{eq:proofsatisfied}
\sum \mu_{i,j} (c_i + \alpha_{i,j}c_j) = \sum_{i=1}^n \lambda_i c_i.
\end{equation}

We can split the sum on the right hand side by considering the positive and negative entries separately,
\begin{equation*}
\sum \mu_{i,j} (c_i + \alpha_{i,j}c_j) = \sum_{i = 1}^\ell \lambda_ic_{i}+\sum_{i = \ell+1}^n \lambda_ic_{i}.
\end{equation*}

For $\lambda_i$, $i \in \{\ell+1, \ldots, n\}$ we can write:
\begin{equation*}
\sum_{k=1}^\ell \mu_{k,i} \alpha_{k,i} = \lambda_i
\end{equation*}
which gives:
\begin{equation}\label{eq:negsum}
\sum_{k=1}^\ell \mu_{k,i} \ (-c_{k,1}) = c_{i,1} \lambda_i ,\: i \in \{\ell+1, \ldots, n\}.
\end{equation}

For $\lambda_i$, $i\in \{1, \ldots, \ell\}$ we can write:
 \begin{equation*}
 \sum_{k=\ell+1}^n \mu_{i,k} = \lambda_i
 \end{equation*}
which gives:
\begin{equation} \label{eq:possum}
 \sum_{k=\ell+1}^n \mu_{i,k} c_{i,1} = c_{i,1}\lambda_i, \:  i \in \{1, \ldots, \ell\}.
\end{equation}

The equations \eqref{eq:negsum}, \eqref{eq:possum} can be written in the following way,
\begin{equation}\label{eq:Axb}
\scalemath{.85}{
\begin{array}{c@{\!\!\!}l}
  \begin{array}[c]{c}
   \left. \begin{array}{c} \vphantom{0}  \\
   \vphantom{0} \\
   \text{$\ell$ rows}\\
   \vphantom{0} \end{array} \right\{
   \\ \left. \begin{array}{c}
   \text{$n-\ell$ rows}  \end{array} \right\{
\end{array}
&
  \left( \begin{array}[c]{cccc}
    [c_{1,1}] & 0 & \cdots & 0 \\
    0 & [c_{2,1}] & \cdots & 0 \\
    \vdots & \vdots&  &\vdots \\
    0 & 0&  \cdots&[c_{\ell,1}] \\
    -diag(c_{1,1}) & -diag(c_{2,1})& \cdots &-diag(c_{\ell,1}) \\

  \end{array}  \right)

\end{array}
\cdot \begin{pmatrix}
 \mu_{1,\ell+1} \\
 \mu_{1,\ell+2} \\
 \vdots \\
 \mu_{1,n} \\
 \mu_{2,\ell+1} \\
 \vdots\\
 \mu_{2,n}\\
 \vdots\\
 \mu_{\ell,n} \\
\end{pmatrix} =
\begin{pmatrix}
 c_{1,1}\lambda_{1} \\
 c_{2,1}\lambda_{2} \\
 \vdots \\
 c_{n,1} \lambda_{n} \\
\end{pmatrix}
}
\end{equation}
where [$c_{i,1}$] represents a $1 \times (n-\ell)$ vector with all entries equal to $c_{i,1}$ and $diag(c_{i,1}) \in M_{(n-\ell) \times (n-\ell)}$ a diagonal matrix with $c_{i,1}$ on the diagonal.
Let $A$, $b$ be the matrix on the left-hand side and
the vector on the right-hand of equation \eqref{eq:Axb}, respectively.
We can then apply Farkas' Lemma. Consider a vector $y\in\R^n$ such that $A^T \cdot y\geq  0$,
then for the first row we have $c_{1,1}y_1 - c_{1,1}y_{\ell+1} \geq 0$. Further for any $i=1,\ldots, \ell$
and
$j=\ell+1,\ldots,n$ we have
$c_{i,1}y_i - c_{i,1}y_{j} \geq 0$
which implies
$y_i\geq y_j$.
Setting $\bar{y} = max(y_{\ell+1} ,\ldots, y_n)$ we have $y_i\geq \bar{y}$ for $i=1,\ldots, \ell$.
Since all $c_{j,1} <0$ for $j=\ell+1,\ldots,n$ and all $\lambda_j \geq 0$ we have that $y_jc_{j,1}\lambda_j \geq \bar{y}c_{j,1}\lambda_j$. Next we define $\bar{\mathbf{y}}=(\bar{y},\ldots, \bar{y})\in \R^n$ and note that $b^T y\geq b^T \bar{\mathbf{y}}$. Since $\sum_{\lambda_i \geq 0} \lambda_ic_{i,1} = 0 $ we have $b^T \bar{\mathbf{y}} = 0$.
Therefore
condition 2 of Farkas' Lemma fails.
This implies that condition 1 is true and thus there exists a set of positive $\mu$'s that solve the system \eqref{eq:Axb}.
Since a set of $\mu$'s can be found we have that after the first iteration of the algorithm the rows of $V_1$ span $W_1$. By similar argument the rows of the $V_i$ in subsequent iterations span the set of vectors $w$ such that the first $i$ entries of $S(x)w$ equal zero. After $n$ iterations, the rows of $V_n$ span $\mathcal{N}(S(x)) \cap (\mathbb{R}_+)^{m}$, thus we are done.
\end{proof}

\begin{example}{Comparison of the standard and positive bases}\label{ex:compare}

Using the RCT network, an example is presented which compares a standard basis of the nullspace and positive basis for the nullspace. Here we provide the calculations needed to show \textit{Statement 2} and \textit{Statement 3} of Proposition \ref{prop:Positive_Basis} for the RCT. The full details for this example are contained in the Appendix.

\end{example}

\section{Equilibria and asymptotic behavior of metabolites for fixed fluxes} \label{sec:fix-f}

In this section we characterize equilibria and the stability of LIFE systems with a fixed flux vector. We start with linear systems,
and then we consider special LIFE systems, with Assumptions (B) and (C).
LIFE systems satisfying at least Assumption (A) are known as \emph{compartmental systems} in the automatic control community \cite{compartmental,Bullo}. We build upon the well-established results on linear compartmental systems to get a full understanding of linear LIFE systems, as well as for special LIFE systems satisfying Assumption (C). 
Notice that weakly connected components of $\tilde G$ correspond to subsystems having no interaction with each other, so that we can study each weakly connected component separately from the others. Hence, throughout this section, we assume $\tilde G$ is weakly connected, without loss of generality.

\subsection{Linear systems} \label{sec:linear}
In this section we study the properties of linear systems,
i.e. special LIFE system \eqref{met-dyn-LIFE} satisfying Assumption~(C) with fluxes $H_v(x_v)=x_v$, or, equivalently, a system \eqref{eq:dyn-J}, with $h(x) =x$.
We first focus on metabolic networks with no intakes nor excretions,
recalling results from a rich literature from different communities.
For the case with intakes and excretions, we make use of results for Compartmental Systems \cite{Bullo}.

\subsubsection{Linear system without intakes and excretions} \label{subsec:lin-nointake-noexcretion}
In the case with no intakes nor excretions (also known as free closed system),
the linear LIFE dynamics become $\dot x = J(f) x$.
Notice that $J(f)$ is a \emph{Metzler matrix} (i.e. has non-negative off-diagonal entries), and all its columns sum to zero ($\mathbf 1^T J(f)  = \mathbf 0^T$).

This system is known as \emph{Laplacian} dynamics, because $L = -J(f)^T$ is a weighted Laplacian of the graph $G$ defined as follows.
Given the graph $G$ and given strictly positive weights $f_e$ associated with its edges
(for us, the weights are the fluxes),
the \emph{weighted adjacency matrix} $A$ is defined by $A_{ij} = f_{(v_i,v_j)}$ if $(v_i,v_j)\in E$, and $A_{ij} = 0$ otherwise.
The corresponding weighted Laplacian is $L = D-A$, where $D = \diag(A \mathbf 1)$ is the diagonal matrix containing the weighted out-degrees, i.e. row-sums of $A$.

The eigenvalues of the Laplacian, and in particular its eigenvalue $0$ and the corresponding eigenspace which is actually the set of equilibria of $\dot x = J(f) x$, have been studied in graph theory (see e.g.~\cite{laplacian-2006}).
The study of the spectrum of the Laplacian has received extensive attention also in the automatic control community,
on one hand because of the recent interest in the so-called consensus dynamics (see \cite[Chapter 7]{Bullo}) $\dot x = - L x$ (which are different from the dynamics $\dot x = J(f) x$, since $L=-J(f)^T$, but the eigenvalues of $-L$ and $J(f)$ are the same),
and on the other hand because of the interest in the linear compartmental system $\dot x = J(f) x$ (see \cite[Chapter 9]{Bullo}).

Laplacian dynamics have been introduced in the mathematical biology literature by \cite{Laplacian}, together with a complete study of their equilibria and convergence properties.
Notice that in \cite{Laplacian} the matrix $J(f)$ itself is called a Laplacian, while in the graph theory and control theory communities the name Laplacian refers to $L=-J(f)^T$.

The dynamics $\dot x = J(f) x$ is also very related to \emph{continuous-time Markov chains} (see e.g.\ the textbook \cite{MarkovChains}). Thus we elaborate on this, a homogeneous continuous-time Markov chain over a finite state $V = \{ 1, \dots, n\}$ is a stochastic process $X(t)$ taking values in $V$,
that satisfies the Markov property:\\
\noindent
for all times $t_0 \le t_1 \le \dots \le t_h \le t_{h+1}$,
\[
\Pr(X_{t_{h+1}}=i_{h+1}
	\mid X_{t_{0}}=i_{0},X_{t_{1}}=i_{1},\ldots ,X_{t_{h}}=i_{h})=
\Pr(X_{t_{h+1}}=i_{h+1} \mid X_{t_{h}}=i_{h}), \]
and which is homogeneous in time:
{\small\[\Pr(X_{t_{h+1}}=i_{h+1} \mid X_{t_{h}}=i_{h}) =
\Pr(X_{t_{h+1}-t_h}=i_{h+1} \mid X_{t_{0}}=i_{h}) =
P_{i_{h}, i_{h+1}}(t_{h+1}-t_h)\,.
\]}$P(t)$ is the matrix whose $(i,j)$-th entry represents the probability of transition from state $i$ to state $j$ in an interval of time of length $t$.
$P(t)$ is assumed to be right-differentiable, and with time derivative defined by
$\dot{P}(t) = \lim_{h \to 0^+} (P(t+h) - P(t))/h$.
The Markov chain is fully described by its \emph{generator matrix} (or \emph{transition rate matrix}) defined by $Q = \dot{P}(0)$.
Indeed, $P(t)$ is related to $Q$ by the \emph{Kolmogorov forward equation} $\dot{P}(t) = P(t) Q$.
The unique solution of the Cauchy problem $\dot{P}(t) = P(t) Q$ with $P(0)=I$ is
$P(t) = e^{Qt}$.
Denoting by $\pi(t)$ a column vector whose $i$th entry is $\pi_i(t) = \Pr(X(t)=i)$,
we have that $\pi^T(t) = \pi^T(0) P(t) = \pi^T(0) e^{Qt}$ and hence is a solution of
$\dot \pi^T = \pi^T Q$ with given initial condition $\pi(0)$.

The generator matrix $Q$ is a Metzler matrix whose rows sum to $0$ ($Q \mathbf 1 = \mathbf 0$),
and hence there is an immediate equivalence between the dynamics
$\dot \pi^T = \pi^T Q$ and
the linear LIFE system $\dot x = J(f) x$, simply by taking $Q = J(f)^T$.
In Markov chains, one looks at $\pi$ being a probability vector, namely having positive entries and $\sum_i \pi_i = 1$. In LIFE systems, we are also interested in a vector $x$ with positive entries, but the total mass could be arbitrary, so that one should set $\pi_i(0) = x_{v_i}(0)/m_0$, with $m_0 = \sum_i x_{v_i}(0)$.

Based on all this rich literature, we recall here the results about equilibria and their stability.
The spectrum of $J(f)$ is characterized as follows (see e.g.\ \cite[Proposition.~$1$]{Laplacian}):
\begin{proposition} \label{prop:Laplacian}
Assume there are no intakes nor excretions,
then:
\begin{itemize}
\item All eigenvalues of $J(f)$ are either $0$ or have strictly negative real part.
\item The dimension of the nullspace of $J(f)$ is equal to the algebraic multiplicity of the $0$ eigenvalue
\footnote{The algebraic multiplicity of an eigenvalue is its multiplicity as a root of the characteristic polynomial, while its geometric multiplicity is the dimension of the corresponding eigenspace. In the case of the zero eigenvalue, the eigenspace is the nullspace of the matrix.},
and is equal to the number of terminal components in $G$.
\item Moreover, denoting by $G_1, \dots, G_k$ the terminal components of $G$, there exists a basis of the nullspace of $J(f)$ composed of vectors $\pi_1, \dots, \pi_k$, such that vector $\pi_i$ has strictly positive entries in correspondence of vertices in $G_i$, and has vanishing entries otherwise.
\end{itemize}
\end{proposition}
It is customary to `normalize' the vectors $\pi_1, \dots, \pi_k$ so that their entries sum to 1, so that they can be interpreted as probability vectors. Under this choice, the restriction of vector $\pi_i$ to the terminal component $G_i$ is the stationary distribution of the Markov chain restricted to such component, i.e. its unique equilibrium with mass $1$ (uniqueness is obtained from the fact that the terminal component is a strongly connected component).

Notice that the dimension of the nullspace of an incidence matrix $\Gamma$ (and hence of a stoichiometrix matrix $S(x)$) is related to the number of weakly connected components, while the dimension of the nullspace of a Laplacian matrix is related to strongly connected components, and not all components matter, but only the terminal ones.

In the case where $G$ has a unique terminal component, the nullspace of $J(f)$ has dimension one: the equilibrium is unique, up to a multiplicative factor which is the initial total mass $m_0 = \sum_i x_{v_i}(0)$.
In the case where $G$ is strongly connected, the whole graph is a single strongly connected component. In this case, not only the nullspace of $J(f)$ has dimension one, but also it is generated by a vector whose entries are all strictly positive.

The spectral properties of $J(f)$ given in Proposition~\ref{prop:Laplacian} fully characterize the asymptotic behavior of the linear system $\dot x = J(f) x$, by standard theory of linear systems, giving the following:
\begin{proposition} \label{prop:asymp-linear-closed}
Consider the linear LIFE system with no intakes nor excretions and denote by $G_1, \dots, G_k$ the terminal components of the associated graph $G$. The following properties hold:
\begin{itemize}
\item The total mass of the system $m = \sum_{v\in V} x_v = m_0$ is constant in time.
\item From any positive initial condition $x(0)$, the system converges to an equilibrium, having strictly positive entries in correspondence of vertices of terminal components, and $0$ elsewhere.
\item Moreover, if there is a unique terminal component, then the equilibrium is uniquely determined by the initial total mass $m_0$.
\end{itemize}
\end{proposition}

\subsubsection{Linear system with intakes and excretions}
We now focus on linear systems with intakes and/or excretions,
using results from linear compartmental systems as summarized in \cite[Chapter 9]{Bullo}.
A first remark is that, having introduced a single $v_0$ from which all intake edges are originated, and a single $v_{n+1}$ to which all excretion edges are headed, we have a single weakly connected component containing intakes and/or excretions.
We restrict our attention to such a weakly connected component, while other possible components with no intakes nor excretions have a behavior described in Sect.~\ref{subsec:lin-nointake-noexcretion}.

The dynamics are given by $\dot x = J(f) x + \phi$, where vector $\phi$ represents the intakes. Due to excretions $J(f)$ does not have all column-sums equal to $0$.
This means that $-J(f)^T$ is not any more a Laplacian, but it is a \emph{grounded Laplacian}.
The term grounded Laplacian refers to a matrix obtained from a larger Laplacian matrix by deleting the row and column corresponding to a given vertex; the name `grounded' has an interpretation for electrical networks, where this corresponds to connecting the given vertex to the ground. Consider the subgraph of $\tilde G$ where we have removed $v_0$ but not $v_{n+1}$, and consider the weighted Laplacian $L \in M_{(n+1) \times (n+1)}$ of such a graph (with weights equal to the fluxes), then define $L_g$ by deleting the last row and last column (associated with $v_{n+1}$). The resulting grounded Laplacian is $L_g$ is such that $J(f) = -L_g^T$.

The spectral properties of $J(f)$ are summarized as follows (see \cite[Theorem $9.5$ and Lemma $9.12$]{Bullo}):
\begin{proposition}\label{prop:invert-grounded-Lap}
Consider a linear system with intakes and/or excretions, then:
\begin{itemize}
\item All eigenvalues of $J(f) $ are either $0$ or have strictly negative real part.
\item The dimension of the nullspace of $J(f)$ is equal to the algebraic multiplicity of the $0$ eigenvalue, and is equal to the number of terminal components not containing any excretion.
\end{itemize}
In particular, the following are equivalent:
\begin{enumerate}[label=(\alph*)]
\item For every $v \in V$ there is a path from $v$ to $X$.
\item $J(f)$ is Hurwitz stable (i.e. all its eigenvalues have strictly negative real part).
\item $J(f)$ is invertible.
\end{enumerate}
Moreover, when $J(f)$ is invertible, all entries of $-J(f)^{-1}$ are positive; if $G$ is strongly connected, they are strictly positive.
\end{proposition}

Equilibria of the dynamics $\dot x = J(f)x + \phi$ are the solutions of the linear system of equations $J(f) x = -\phi$.
By Proposition~\ref{prop:invert-grounded-Lap}, if all vertices $v$ have a path to some excretions, then there is a unique equilibrium $\bar x = - J(f)^{-1} \phi$, and moreover all entries of $\bar x$ are positive.
To study the general case, where vertices might or might not have a path to some excretions, it is convenient to partition the system into two subsystems, as follows.
Partition the vertex set as  $V = V_{1} \cup V_2$, with $V_1$ the set of $v \in V$ such that there is a path from $v $ to $v_{n+1}$.
Without loss of generality, we can re-label vertices in $V$ so as to have vertices $v_1, \dots, v_r \in V_1$ and $v_{r+1}, \dots, v_n \in V_2$. According to this decomposition, partition the vector $x$ into two blocks $x_1$ corresponding to $V_1$ and $x_2$ corresponding to $V_2$, and similarly partition $\phi$ as $\phi_1$, $\phi_2$.
Notice that there is no edge from $V_2$ to $V_1$,
since an edge $(w,v)$ with $v \in V_1$ implies that there is a path from $v$ to $v_{n+1}$ and also a path from $w$ to $v_{n+1}$.
Hence, we have:
\[ J(f) = \left[ \begin{matrix}
J_{1} & 0 \\
J_{21} & J_2
\end{matrix} \right] ,\]
and
\begin{equation} \label{eq:reduced}
\begin{cases}
\dot x_1 = J_1 x_1 + \phi_1 \\
\dot x_2 =  J_2 x_2  + J_{21} x_1 + \phi_2.
\end{cases}
\end{equation}
The first subsystem is called the \emph{reduced system}, and its evolution is not affected by the second subsystem.
The matrix $J_1$ is equal to the matrix $J(f)$ of the graph $\tilde H$ obtained from $\tilde G$ as follows:
for any edge $(v,w)$ with $v \in V_1$, $w \in V_2$, remove the edge $(v,w)$ and add the edge $(v, v_{n+1})$ with flow $f_{v, v_{n+1}} = f_{v,w}$;
then remove all vertices in $V_2$ and all corresponding edges.
By definition of $V_1$, all vertices of $\tilde H$ have a path to $v_{n+1}$, and hence, by Proposition \ref{prop:invert-grounded-Lap}, $J_1$ is invertible and Hurwitz stable, and $-J_1^{-1}$ has positive entries.

For the second subsystem, it is easy to see that $J_2$ is the matrix $J(f)$ of the subgraph $K$ of $G$ corresponding to vertices in $V_2$.
Hence, $L = -J_2^T$ is a Laplacian matrix, and by Proposition \ref{prop:Laplacian} its nullspace is generated by vectors $\pi_1, \dots, \pi_k$,
where $G_1, \dots, G_k$ are the terminal components in $K$ (same as the terminal components with no excretions  in $G$),
and each vector $\pi_i$ has strictly positive entries corresponding to vertices in $G_i$ and is $0$ elsewhere.
Choosing each $\pi_i$ so that its entries sum to 1, the restriction of $\pi_i$ to $G_i$ is the stationary distribution of the corresponding Markov chain restricted to $G_i$, namely the Markov chain whose generator matrix is the transpose of the submatrix of $J(f)$ corresponding to $G_i$.

These remarks, together with standard tools of analysis of linear dynamical systems, lead to the following proposition:
\begin{proposition}[\cite{Bullo}, Theorem~$9.13$] \label{prop:asymptotics-linear}
Consider a weakly connected linear LIFE system with positive flows on all edges of $\tilde G$.
From any positive initial condition, the reduced system (namely the subsystem connected to $X$) converges to its unique equilibrium with positive entries $\bar{x}_1 = -J_1^{-1} \phi_1$.

If there are some vertices not connected to $X$, then:
\begin{itemize}
\item If there is a terminal component $G_T$ with no excretions such that there is a path from $I$ to $G_T$, then the mass in $G_T$ grows unbounded,
and hence also $\lim_{t \to \infty}  \| x_2(t) \|= + \infty$ and $\lim_{t \to \infty}  \|  x(t) \|= + \infty$.
\item If for all terminal components with no excretions $G_1, \dots, G_k$ there is no path from $I$ to $G_i$, then the mass of the system remains bounded, and moreover $\lim_{t \to \infty} x_2(t)$ is some equilibrium point $\bar x_2$ (depending on the initial condition), such that all entries of $\bar x_2$ corresponding to non-terminal components are $0$, while the restriction of $\bar x_2$ to a terminal component $G_i$ is proportional to the stationary distribution of the Markov chain with generator matrix $Q$ equal to the transpose of the submatrix of $J(f)$ corresponding to $G_i$.
\end{itemize}

\begin{figure}[h]
\begin{center}
\includegraphics[scale = .20]{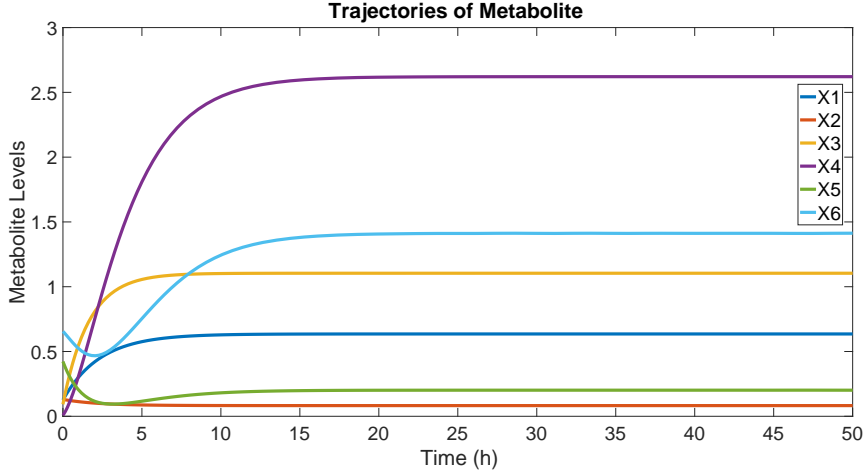}\quad
\caption{The trajectories of the values of metabolites over 25 hours.}
\label{fig:sim_eq}
\end{center}
\end{figure}

\begin{example}\label{ex:sim_eq}
In this example a simulation is used to verify the calculated equilibrium found using Proposition \ref{prop:asymptotics-linear}. Here we use the RCT network shown in \ref{fig:RCT Network} and calculate the equilibrium. Because every vertex of the RCT is connected to $X$ we have that $\dot{x} = J(f)x + \phi$ with,
{\footnotesize\begin{align}
J(f)&=
\left(\begin{array}{cccccc} -f_{\mathrm{(v_1,v_4)}} & 0 & 0 & 0 & 0 & 0\\ 0 & -f_{\mathrm{(v_2v_4)}} & 0 & 0 & 0 & 0\\ 0 & 0 & -f_{\mathrm{(v_3v_4)}} & 0 & 0 & 0\\ f_{\mathrm{(v_1v_4)}} & f_{\mathrm{(v_2v_4)}} & f_{\mathrm{(v_3v_4)}} & -f_{\mathrm{(v_4v_5)}}-f_{\mathrm{(v_4v_6)}} & 0 & 0\\ 0 & 0 & 0 & f_{\mathrm{(v_4v_5)}} & -f_{\mathrm{(v_5v_6)}} & 0\\ 0 & 0 & 0 & f_{\mathrm{(v_4v_6)}} & f_{\mathrm{(v_5v_6)}} & -f_{\mathrm{(v_6v_{n+1})}} \end{array}\right),\nonumber\\ \phi& = \left(\begin{array}{c} f_{\mathrm{(v_0,v_1)}} \\ f_{\mathrm{(v_0,v_2)}}\\f_{\mathrm{(v_0,v_3)}}\\ 0 \\0 \\0\\ \end{array}\right).
\end{align}}

The vector of fluxes and initial metabolite values were randomized to obtain
\begin{equation}
f=
\left(\begin{array}{c} f_{\mathrm{(v_0,v_1)}} \\
f_{\mathrm{(v_0,v_2)}}\\
f_{\mathrm{(v_0,v_3)}}\\
f_{\mathrm{(v_1,v_4)}}\\
f_{\mathrm{(v_2,v_4)}}\\
f_{\mathrm{(v_3,v_4)}}\\
f_{\mathrm{(v_4,v_5)}}\\
f_{\mathrm{(v_5,v_6)}}\\
f_{\mathrm{(v_5,v_6)}}\\
f_{\mathrm{(v_6,v_{n+1})}}\\ \end{array}\right) = \left(\begin{array}{c} 0.2729\\    0.0372 \\   0.6733 \\   0.4296 \\   0.4517 \\   0.6099  \\  0.0594   \\ 0.3158 \\   0.7727\\    0.6964 \\\end{array}\right), x_0 = \left(\begin{array}{c} 0.1253\\
    0.1302\\
    0.0924\\
    0.0078\\
    0.4231\\
    0.6556 \\  \end{array}\right).
\end{equation}

Using these values the equilibrium was calculated to be $\bar{x} = (0.6354,0.0824,\break
    1.1040,
    2.6211,
    0.2015,
    1.4122)$. The RCT network with the randomized initial values was simulated for 50 hours and the simulation results closely matched $\bar{x}$. The simulation results for the first 25 hours are shown in figure \ref{fig:sim_eq}.

\end{example}
\end{proposition}

\subsection{Special LIFE systems}

In this section, we focus on special LIFE systems. We first recall some interesting results from \cite{Maeda} valid under Assumption~(B), and then we exploit them together with the spectral properties of $J(f)$ described in Sect.~\ref{sec:linear} in order to fully characterize equilibria and convergence of special LIFE systems under Assumption~(C). We refer the reader to \cite{compartmental} for some results valid under great generality, with assumptions less restrictive than Assumption~(B).
As for linear systems, we first notice that weakly connected components correspond to subsystems which have no influence on each other, and hence can be studied separately; if there are any weakly connected components with no intake nor excretion, they are subsystems with constant total mass.

\begin{proposition} (\cite[Theorem~6]{Maeda}) \label{prop:Maeda-closed}
Consider the special LIFE system under Assumption (B) with no intakes nor excretions. The following properties hold:
\begin{itemize}
\item The total mass of the system $m = \sum_{v\in V} x_v$ is constant in time.
\item From any positive initial condition $x(0)$, the system tends to the equilibrium set. 
\item Moreover, if there is a unique terminal component, then there exists a unique equilibrium with positive entries with the same mass as the initial mass, and the system converges to it.
\end{itemize}
\end{proposition}

For the general case with intakes and/or excretions, the following result holds on the asymptotic behavior of the dynamics.
\begin{proposition} (\cite[Theorems~2 and 3]{Maeda})  \label{prop:Maeda2-3}
For special LIFE systems under Assumption (B), with a positive initial condition,
\begin{itemize}
\item Trajectories are bounded if and only if there exists an equilibrium with positive entries;
\item If trajectories are bounded, then they approach an equilibrium set for $t \to \infty$, and if moreover the equilibrium set consists of isolated points, then they converge to some equilibrium.
\end{itemize}
\end{proposition}
A caveat reported in \cite{Maeda} is that the second item in Proposition~\ref{prop:Maeda2-3} does not rule out the possibility to have non-periodic oscillatory trajectories that approach the equilibrium set lying outside it, and approaching it rotating infinitely many times; this can happen in the case of a connected compact equilibrium set.

\begin{remark}
Recall that, by Proposition~\ref{prop:A_neccessary}, if there exists a vertex $v\in V$ such that there is a path from $I$ to $v$ and no path from $v$ to $X$, then there exists no equilibrium. By the first item of Proposition~\ref{prop:Maeda2-3}, this further implies that all trajectories are unbounded.
\end{remark}

The following results concern the existence and uniqueness of equilibria.
\begin{proposition}(\cite[Theorems~4 and 5]{Maeda}) \label{prop:Maeda4-5}
For special LIFE systems under Assumption (B), the following holds.
\begin{itemize}
\item There exists an equilibrium with positive entries for arbitrary constant intakes if and only if for all $v \in V$ there is a path to $X$ such that all edges in the path have $\lim_{x_v\to \infty} H_e(x_v) = +\infty$;
\item If there exists an equilibrium with positive entries, and if there exists a path from all $v \in V$ to $X$, then the equilibrium is unique.
\end{itemize}
\end{proposition}

Notice that in the first item, the `only if' part is true only when we require existence of equilibria with positive entries for completely arbitrary intakes: arbitrary intake set $I$ and arbitrary positive values of the corresponding fluxes; this strong condition is not necessary to have an equilibrium with positive entries for a given set $I$ and a given value of the intake fluxes.
Also notice that this statement considers fixed fluxes $f_e$, differently from Proposition~\ref{prop:rankS2}, where $x$ is fixed and fluxes $f_e$ are allowed to vary (except for intake fluxes); the crucial difference is that the product $H_e(x) f_e$ can be made arbitrarily large in the context of Proposition~\ref{prop:rankS2}, even in the case where $H_e(x)$ is bounded.

A remark about the second item is that, by Proposition~\ref{prop:Maeda2-3}, the existence and uniqueness of the equilibrium with positive entries further implies that all trajectories converge to such equilibrium.

To apply the first item of Proposition~\ref{prop:Maeda2-3} or the second item of Proposition~\ref{prop:Maeda4-5}, one needs to already have the knowledge about existence of an equilibrium with positive entries. This happens for example in the case where one starts by fixing a desired equilibrium with positive entries $\bar x$, and then applies the extreme pathways technique in order to design suitable fluxes ensuring that $\bar x$ is an equilibrium of the system. Under some assumptions on the graph, the above propositions then ensure uniqueness of the equilibrium, and its global asymptotic stability.

\subsubsection{Special LIFE systems under Assumption (C)}
In the remainder of this section, we consider Assumption (C). In this case, recall that the dynamics can be re-written as $\dot x = J(f) h(x) + \phi$. Different from linear systems, $h(x)$ can contain non-linearities, but the matrix $J(f)$ has the same definition as for linear systems, so that its spectral properties are described by Propositions~\ref{prop:Laplacian} and~\ref{prop:invert-grounded-Lap}.
Also notice that equilibria, i.e. solutions of $J(f) h(x) = - \phi$, can be found by solving $J(f) h = -\phi$, where $h $ is an unknown vector in $(\R_+)^n$, and then solving $h(x) = h$. The latter is equivalent to solving $H_{v_i}(x_{v_i}) = h_i$, for $i = 1,\dots, n$.

We start by studying the case with no intakes nor excretions. In this case, by Proposition~\ref{prop:Laplacian},
$J(f) h = 0$ means $h \in \mathcal N(J(f))$, where the nullspace $\mathcal N(J(f))$ is generated by $\pi_1, \dots, \pi_k$, associated with terminal components $G_1, \dots, G_k$.
Without loss of generality, we re-label vertices so that the first $n_1$ vertices encompass the first terminal component $G_1$, the following $n_2$ vertices encompass the second terminal component $G_2$, and so on, up to the last $n_k$ vertices encompass the last terminal component $G_k$, and finally the remaining vertices are not in any terminal component (say there are $n_0$ of them).
Denote the corresponding subblocks of vector $h$ as $h^{(1)}, \dots, h^{(k)}$ and $h^{(0)} $ for the non-terminal ones, and similarly define $x^{(1)}, \dots, x^{(k)}$ and $x^{(0)}$ for vector $x$.
By Proposition~\ref{prop:Laplacian}, the nullspace of $J(f)$ is
$\mathcal N(J(f)) = \{h \in (\R_+)^n : \, h^{(j)} = \alpha_j \tilde \pi_j, \, \alpha_j \in \R \text{ and } h^{(0)} = 0\, \text{, for } j \in \{1,\ldots, k\} \}$, where $\tilde \pi_i $ is the restriction of $\pi$ to the component $G_i$, i.e. is a vector of size $n_i$ with strictly positive entries, representing the stationary distribution of the Markov chain whose generator is the transpose of the restriction of $J(f)$ to $G_i$.
Now we need to characterize the set of equilibria with positive entries $\bar X := \{x \in (\R_+)^n \text{ such that } h(x) \in \mathcal N(J(f))\}$.
Recall that $H_{v}(x_{v})$ are strictly increasing functions, being $0$ when $x_v=0$.
Denote by $R_v$ the range of $H_v$ (for $x_v \ge 0$), notice that either $R_v = [0, h^{\max})$, or $R_v = [0, +\infty)$.
Now denote by $\mathcal H_j$ the set of vectors $\alpha_j \tilde \pi_j$ such that $\alpha_j \ge 0$ and  $[\alpha_j \tilde \pi_j]_{v} \in R_v$ for all $v$ in $G_j$.
Then denote by $H_j^{-1}h^{(j)}$ the vector obtained from $h^{(j)} \in \mathcal H_j$ by applying entry-wise the inverse functions $H_v^{-1}$.
Finally we obtain
$\bar X= \{ x \in (\R_+)^n \text{ such that } x^{(j)} = H_j^{-1}h^{(j)}, h^{(j)} \in \mathcal H_j \text{ and } h^{(0)} = 0  \}$.

Having characterized the set of equilibria with positive entries, now recall that Proposition~\ref{prop:Maeda-closed} applies, and trajectories remain bounded, with total mass constant in time, and approach the above-described equilibrium set.
We focus on the case with intakes and/or excretions. The definition of the reduced system and the partitioning in two subsystems introduced for linear systems applies also to special LIFE systems under Assumption (C), with the only difference that now the dynamics are non-linear:
\[
\begin{cases}
\dot x_1 = J_1 h_1(x_1) + \phi_1 \\
\dot x_2 =  J_2 h_2(x_2)  + J_{21} h_1(x_1) + \phi_2.
\end{cases}
\]
vectors $h_1(x_1)$ and $h_2(x_2)$ having replaced $x_1$ and $x_2$ in \eqref{eq:reduced}.

Hence, one can obtain the following analogous of Proposition~\ref{prop:asymptotics-linear}.
\begin{proposition}
Consider a weakly connected special LIFE system satisfying Assumption~(C).
Denote by $R_v $ the range of the function $H_v(x_v)$ (for $x_v \ge 0$)
and define $\bar h = -J_1^{-1} \phi_1$.
If $\bar h_v \in R_v$ for all $v \in V_1$, then
from any initial condition, the reduced system converges to its unique equilibrium $\bar x_1$ defined by $[\bar x_1]_v = H_v^{-1} (\bar h_v)$.
Otherwise, the system has no equilibrium, and $\lim_{t \to \infty} \| x(t) \| = + \infty$.

In the case where $x_1(t)$ converges to $\bar x_1$, if there are some vertices not connected to $X$, then:
\begin{itemize}
\item If there exists a terminal component $G_T$ with no excretions such that there is a path from $I$ to $G_T$, then the mass in $G_T$ grows unbounded,
and hence also $\lim_{t \to \infty} \| x_2(t) \| = + \infty$ and $\lim_{t \to \infty }\| x(t) \|= + \infty$;
\item If for all terminal components with no excretions $G_1, \dots, G_k$ there is no path from $I$ to $G_i$, then the mass of the system remains bounded, and moreover entries of $x_2(t)$ corresponding to non-terminal components converge to zero, while the restriction of $\bar x_2$ to a terminal component $G_i$ approaches the equilibrium set constructed as follows.
The restriction of $J_2$ to vertices in $G_i$ has a nullspace generated by a single positive vector $\tilde \pi_i$;
let $H_i$ denote the subset of such nullspace given by vectors $h = \alpha \tilde{\pi}_i$ such that $h_v \in R_v$ for all vertices $v$ of $G_i$;
the equilibrium set is given by $x$'s such that there exist $h \in H_i$ verifying $x_v = H_v^{-1}(h_i)$ for all vertices $v$ of $G_i$.
\end{itemize}
\end{proposition}

\begin{proof}
Recall that $x_1$ is an equilibrium for $\dot x_1 = J_1 h_1(x_1) + \phi_1$ if and only if $J_1 h_1(x_1) = -\phi_1$.
By Proposition \ref{prop:invert-grounded-Lap}, $J_1$ is invertible and $-J_1^{-1}$ has positive entries; define $\bar h$ to be the unique solution to $J_1 h_1 = -\phi_1$. Notice that $\bar{h}$ has all positive entries.
If all entries of $\bar h$ are within the range of the corresponding function $H_v(x_v)$, then we have a unique equilibrium with positive entries $\bar x_1$ obtained as in the statement of the proposition, and hence by Proposition~\ref{prop:Maeda2-3} $x_1(t)$ converges to $\bar x$.
Otherwise, there exists no equilibrium, and hence, by Proposition~\ref{prop:Maeda2-3}, the mass of system grows indefinitely.

If there is a terminal component with no excretions but connected to some intakes, then by Proposition~\ref{prop:asymptotics-linear} there is no equilibrium, and hence by Proposition~\ref{prop:Maeda2-3} the mass of system grows indefinitely.
If all terminal components with no excretions are not connected to intakes, then the equilibrium set is obtained from the properties of the nullspace of $J_2$, given by Proposition~\ref{prop:Laplacian}. Moreover, Proposition~\ref{prop:Maeda2-3} ensures that trajectories remain bounded and approach the equilibrium set.
\end{proof}

\subsection{Zero-deficiency theory}
In this section, we shortly recall the zero-deficiency theory, and compare it with our results on equilibria of LIFE systems.
Zero-deficiency theory arises in the literature on chemical reaction networks, a seminal paper is \cite{Zero-Deficiency}. Our short overview is based on \cite{zero-def-notes}.

A free closed chemical reaction network with $m$ reactions between $p$ complexes involving $n$ species can be described by
$\dot x = S \Gamma R(x)$ where $S \in M_{n\times p}$, $\Gamma \in M_{p \times m}$ is the incidence matrix of the network, and $R(x)$ a column vector of size $m$.
The \emph{deficiency} of the chemical reaction network is the difference of the dimensions of the nullspaces of $S \Gamma$ and $\Gamma$:
$\delta  = \dim \ker (S \Gamma) - \dim \ker (\Gamma)$, which is equivalent to
$\delta = \rank (\Gamma) - \rank(S\Gamma) $ and to
$\delta = p - \ell - \rank (S\Gamma)$, where $\ell$ is the number of weakly connected components. The first equivalence is due to rank-nullity Theorem, and the second to the fact that $ \rank (\Gamma) = p - \ell$, since $\Gamma$ is an incidence matrix (Proposition 4.3 of \cite{Biggs}).
The reaction rate vector $R(x)$ is governed by the `mass-action kinetics': $R(x) = K \Psi(x)$, where $K \in M_{m \times p}$ and $\Psi(x)$ is a vector of size $p$,  defined as follows:
$K_{ej} = k_e$ if complex $j$ is the reactant complex of reaction $e$, $K_{ij}=0$ otherwise; $\Psi_i(x) = \prod_{j=1}^n  x_j^{s_{ji}}$.
Hence, the dynamics can be equivalently re-written as
$\dot x = S \Gamma K \Psi(x)$.
The following results hold:
\begin{proposition}[Zero-deficiency Theorem]
Consider a free closed chemical reaction network with mass-action kinetics.
If its deficiency is zero, then:
there exists an equilibrium with strictly positive entries if and only if the system is weakly reversible (i.e. each weakly connected component is also strongly connected).

Moreover, this strictly positive equilibrium is unique in each stoichiometric class (i.e. each weakly connected component has a space of equilibria which has dimension one, so that its equilibrium is unique up to a multiplicative constant representing the total mass in the component), and it is locally asymptotically stable.
\end{proposition}

Notice the close resemblance with the results for linear LIFE systems in the case with no intakes nor excretions, described in Propositions~\ref{prop:Laplacian} and~\ref{prop:asymp-linear-closed}.
It turns out that a particular class of closed free chemical reaction networks with mass-action kinetics and zero deficiency exactly coincides with linear LIFE systems with no intakes nor excretions. Indeed, let $p=n$, let $S$ be the identity matrix of size $n$, let $K$ be defined by $K_{ej} = f_e$ if edge $e$ is of the form $e=(v_j, w)$ and $K_{ej}=0$ otherwise; moreover choose exponents $s_{ji}$ to be zero and ones so that $\Psi(x) = x$. This gives exactly a linear LIFE system with no intakes nor excretions. It is immediate to see that the deficiency is zero, since $S = I$ implies $\rank(S\Gamma) = \rank(\Gamma)$.

In the more general case, free closed chemical reaction networks with mass-action kinetics and zero deficiency are a class of LIFE systems with different assumptions than those considered in this paper, e.g., they might not even satisfy Assumption~(A), due to the presence of a rectangular matrix $S$ left-multiplying $\Gamma$, and often do not satisfy Assumption~(B), due to the dependence of entries of $\Psi(x)$ on various entries of $x$.

\section{Conclusion} \label{sec:conclusion}
For general LIFE systems (Assumption (A)), we show the existence of positive solutions.  Stability criteria have been shown for the structure of a graph associated to LIFE systems.  We show that further conclusions cannot be drawn in the general case for LIFE systems because arbitrary dynamics may be defined.

For the problem of understanding equilibria for a fix set of metabolites, we show that the rank of the stoichiometric matrix (and the nullspace of this matrix) are determined by structural properties of the graph associated to the system, discuss the effect of the associated graph having intake vertices and excretion vertices on the rank of the stoichiometric matrix, and give necessary conditions on the structure of the graph associated to LIFE systems for the existence of equilibria.  More biologically relevant equilibria of LIFE systems are those with all positive metabolites and fluxes and we prove that the extreme pathways method for finding a positive basis describes all such equilibria.

The field of network flows contributes to LIFE systems by way of the min cut max flow theorem.  We show the capacity of edges in network flows relates to saturation of functions corresponding to edges of LIFE systems.  The method of extreme pathways to calculate a positive basis is proven to include the entire intersection of the nullspace with the positive orthant.  This basis is essential to describing equilibria of LIFE systems.

Equilibria and asymptotic behavior of metabolites for fixed fluxes are studied.  We show that under stricter assumptions (Assumption (C)) we determine the eigenvalues of the jacobian of the LIFE system as well as structure of the graph which admit certain equilibria.  We analyze the case with intakes and excretions versus the cases without intakes or excretions.

Equilibria for LIFE systems with terminal components exist, but these equilibria are dependent upon initial mass of the system.  Furthermore, conditions are given for LIFE systems to tend to equilibrium from non-negative initial conditions.  LIFE systems under Assumption (B) have bounded solutions if and only if there exits a non-negative equilibrium.  Then for Assumption (B), we give conditions on the structure of the associated graph for which nontrivial equilibria exist.

We show that LIFE systems with no intakes or excretions have ``zero deficiency.''  The rank of the stoichiometric matrix (defined in this work) gives information about the structure of the associated graph, specifically the connectivity of the graph and the existence of strongly connected components.  Zero deficiency theory also provides information about the existence of equilibria with respect to this structure.

The structural conditions of graphs discussed in this work have implications on metabolic networks.  We have included results about networks which are considered non-biological for the sake of completeness.  These results allow a clear picture of the structure of metabolic systems which are capable of admitting a biologically relevant equilibrium.  With this clear picture, we are able to determine metabolic networks for which it is most advantageous to analyze using Linear-In-Flux-Expressions.
\section*{Appendix}
This Appendix contains the calculations for examples \ref{ex:extremes} and\ref{ex:compare}. Example \ref{ex:extremes} illustrates finding the extreme pathways of RCT using the algorithm described in Section \ref{sec:algorithm}. Example \ref{ex:compare} compares a standard basis and positive basis for RCT and shows that when only the positive orthant is considered they describe the same space.

\vspace*{4pt}\noindent\textbf{Example \ref{ex:extremes} continued.}
 As shown previously, the stoichiometric matrix for the RCT is given by:
\begin{equation*}\label{eq:stoichiometric_matrix}
\scalemath{.93}{
S(x)=\left(\begin{array}{cccccccccc}
 1 & 0 & 0 & -x1 & 0 &0 &0&0&0& 0\\
  0 & 1 & 0 & 0 & -x2 &0 &0&0&0&0\\
 0 & 0 & 1 & 0 & 0 &-x3 &0&0&0&0\\
 0 & 0 & 0 & x1 & x2 &x3 &-x4&-x4&0&0\\
 0 & 0 & 0 & 0 & 0 &0 &x4&0&-x5&0\\
 0 & 0 & 0 & 0 & 0 &0 &0&x4&x5&-x6\\
\end{array}\right)
}
\end{equation*}

First let $V_0$ be the $m\times m$ identity matrix where $m$ is the number of columns in $S$. Then define $C_0$ as follows, $V_0 \cdot S^T = C_0$. To better illustrate the row combinations in the algorithm a column is appended to $C_0$ which labels the rows.

We have
\begingroup
\renewcommand*{\arraystretch}{1}
\begin{equation*}
\scalemath{.86}{
\left(\begin{array}{cccccccccc}
 1& 0 &0  &0  &0  &0 &0&0&0&0   \\
  0& 1&  0&  0&  0& 0&0&0&0&0   \\
  0& 0 & 1&0  &0  &0 &0&0&0&0   \\
  0&  0&  0& 1 &  0& 0&0&0&0&0  \\
  0&  0&  0&  0& 1 & 0&0&0&0&0  \\
  0&  0&  0&  0&  0&1 &0&0&0&0  \\
   0&  0&  0&  0&  0& 0&1&0&0&0 \\
  0&  0&  0&  0&  0& 0&0&1&0&0  \\
  0&  0&  0&  0&  0& 0&0&0&1&0  \\
  0&  0&  0&  0&  0& 0&0&0&0&1   \\
\end{array}\right)\cdot S^T = \left(\begin{array}{ccccccc}
1  &0   &0   &0   &0   & 0&R1\\
0  &1   &0   &0   &0   & 0&R2\\
0  &0   &1   &0   &0   & 0&R3\\
-x_1&0   &0   &x_1 &0   & 0&R4\\
0   &-x_2&0   &x_2 &0   & 0&R5\\
0   &0   &-x_3&x_3 &0   & 0&R6\\
0   &0   &0   &-x_4&x_4 & 0&R7\\
0   &0   &0   &-x_4&0   & x_4&R8\\
0   &0   &0   &0   &-x_5& x_5&R9\\
0  &0   &0   &0   &0   & -x_6&R10\\
\end{array}\right)
}
\end{equation*}
\endgroup

The next step is to identify any columns of $C_0$ that have no sources or sinks. For this example there are two columns, column $4$ and column $5$. Then we modify both $V_0$ and $C_0$ by first copying all the rows with a zero in the entry of the first identified column (column $4$), then of the remaining rows add all possible positive combinations of two rows which produce a zero in this column. Note that $V_0$ is $10\times 10$ and $C_0$ is $10\times 6$.

This gives us
\begingroup
\renewcommand*{\arraystretch}{1}
\begin{equation*}
\scalemath{.80}{
\left(\begin{array}{cccccccccc}
 1&0  &0  &0  &0  &0 &0&0&0&0   \\
  0& 1&0  &0  &0  &0 &0&0&0&0   \\
  0&  0& 1&0  &0  &0 &0&0&0&0   \\
  0&  0&  0& \frac{x_4}{x_1} & 0 & 0&1&0&0&0  \\
  0&  0&  0& \frac{x_4}{x_1} & 0 & 0&0&1&0&0  \\
  0&  0&  0&  0& \frac{x_4}{x_2} & 0&1&0&0&0  \\
  0&  0&  0&  0& \frac{x_4}{x_2} & 0&0&1&0&0  \\
  0&  0&  0&  0&  0&\frac{x_4}{x_3} &1&0&0&0  \\
  0&  0&  0&  0&  0&\frac{x_4}{x_3} &0&1&0&0  \\
  0&  0&  0&  0&  0& 0&0&0&1&0                 \\
  0&  0&  0&  0&  0& 0&0&0&0&1   \\
\end{array}\right)\cdot S^T = \left(\begin{array}{ccccccc}
1  &0   &0   &0   &0   & 0&R1\\
0  &1   &0   &0   &0   & 0&R2\\
0  &0   &1   &0   &0   & 0&R3\\
-x_4&0   &0   &0   &x_4 & 0&\frac{x_4}{x_1}R4 + R7\\
-x_4&0   &0   &0   &0   & x_4&\frac{x_4}{x_1}R4 + R8\\
0  &-x_4&0   &0   &x_4 & 0&\frac{x_4}{x_2}R5 + R7\\
0  &-x_4&0   &0   &0   & x_4&\frac{x_4}{x_2}R5 + R8\\
0  &0   &-x_4&0   &x_4 & 0&\frac{x_4}{x_3}R6 + R7\\
0  &0   &-x_4&0   &0   & x_4&\frac{x_4}{x_3}R6 + R8\\
0  &0   &0   &0   &-x_5& x_5&R9\\
0  &0   &0   &0   &0   & -x_6&R10\\
\end{array}\right)
}
\end{equation*}
\endgroup
Note that $V_1$ is $11\times 10$ and $C_1$ is $11\times 6$.

Now doing the same for column $5$ gives
\begingroup
\renewcommand*{\arraystretch}{1}
\begin{equation*}
\scalemath{.75}{
\left(\begin{array}{cccccccccc}
1&0  &0  &0  &0  &0 &0&0&0&0   \\
  0& 1&0  &0  &0  &0 &0&0&0&0   \\
  0&  0& 1&0  &0  &0 &0&0&0&0   \\
  0&  0&  0& \frac{x_4}{x_1} &  0&0 &1&0&\frac{x_4}{x_5}& 0\\
  0&  0&  0& \frac{x_4}{x_1} &  0&0 &0&1&0&0  \\
  0&  0&  0&  0& \frac{x_4}{x_2} & 0&1&0&\frac{x_4}{x_5}&0  \\
  0&  0&  0&  0& \frac{x_4}{x_2} &0 &0&1&0&0\\
  0&  0&  0&  0&  0&\frac{x_4}{x_3} &1&0&\frac{x_4}{x_5}&0  \\
  0&  0&  0&  0&  0&\frac{x_4}{x_3} &0&1&0&0  \\
  0&  0&  0&  0&  0& 0&0&0&0&1   \\
\end{array}\right)\cdot S^T = \left(\begin{array}{ccccccc}
1  &0   &0   &0   &0   & 0&R1\\
0  &1   &0   &0   &0   & 0&R2\\
0  &0   &1   &0   &0   & 0&R3\\
-x_4&0   &0   &0   &0   & x_4&\frac{x_4}{x_1}R4 + R7 +\frac{x_4}{x_5}R9\\
-x_4&0   &0   &0   &0   & x_4&\frac{x_4}{x_1}R4 + R8\\
0  &-x_4&0   &0   &0   & x_4&\frac{x_4}{x_2}R5 + R7 +\frac{x_4}{x_5}R9\\
0  &-x_4&0   &0   &0   & x_4&\frac{x_4}{x_2}R5 + R8\\
0  &0   &-x_4&0   &0   & x_4&\frac{x_4}{x_3}R6 + R7 +\frac{x_4}{x_5}R9\\
0  &0   &-x_4&0   &0   & x_4&\frac{x_4}{x_3}R6 + R8\\
0  &0   &0   &0   &0   & -x_6&R10\\
\end{array}\right)
}
\end{equation*}
\endgroup
Note that $V_2$ is $10\times 10$ and $C_2$ is $10\times 6$.

The same process is then followed for columns containing sources and sinks. Starting at the rightmost column add all possible positive combinations of two rows that create a $0$ entry in this column. Do the same operations to both $V_2$ and $C_2$. Doing this for just the rightmost column we have

\ \vspace*{-10pt}
\begingroup
\renewcommand*{\arraystretch}{1}
\begin{equation*}
\scalemath{.68}{
\left(\begin{array}{cccccccccc}
 1&0  &0  &0  &0  &0 &0&0&0&0   \\[0.5mm]
  0& 1&0  &0  &0  &0 &0&0&0&0   \\[0.5mm]
  0&  0& 1&0  &0  &0 &0&0&0&0   \\[0.5mm]
  0&  0&  0& \frac{x_4}{x_1} &  0& 0&1&0&\frac{x_4}{x_5}&\frac{x_4}{x_6}  \\[0.5mm]
  0&  0&  0& \frac{x_4}{x_1} &  0& 0&0&1&0&\frac{x_4}{x_6}  \\[0.5mm]
  0&  0&  0&  0& \frac{x_4}{x_2} & 0&1&0&\frac{x_4}{x_5}&\frac{x_4}{x_6}  \\[0.5mm]
  0&  0&  0&  0& \frac{x_4}{x_2} & 0&0&1&0&\frac{x_4}{x_6}  \\[0.5mm]
  0&  0&  0&  0&  0&\frac{x_4}{x_3} &1&0&\frac{x_4}{x_5}&\frac{x_4}{x_6}  \\[0.5mm]
  0&  0&  0&  0&  0&\frac{x_4}{x_3} &0&1&0&\frac{x_4}{x_6}  \\[0.5mm]
\end{array}\right)\cdot S^T = \left(\begin{array}{ccccccc}
1  &0   &0   &0   &0   & 0&R1\\[0.5mm]
0  &1   &0   &0   &0   & 0&R2\\[0.5mm]
0  &0   &1   &0   &0   & 0&R3\\[0.5mm]
-x_4&0   &0   &0   &0   & 0&\frac{x_4}{x_1}R4 + R7 +\frac{x_4}{x_5}R9+\frac{x_4}{x_6}R10\\[0.5mm]
-x_4&0   &0   &0   &0   & 0&\frac{x_4}{x_1}R4 + R8+\frac{x_4}{x_6}R10\\[0.5mm]
0  &-x_4&0   &0   &0   & 0&\frac{x_4}{x_2}R5 + R7 +\frac{x_4}{x_5}R9+\frac{x_4}{x_6}R10\\[0.5mm]
0  &-x_4&0   &0   &0   & 0&\frac{x_4}{x_2}R5 + R8+\frac{x_4}{x_6}R10\\[0.5mm]
0  &0   &-x_4&0   &0   & 0&\frac{x_4}{x_3}R6 + R7 +\frac{x_4}{x_5}R9+\frac{x_4}{x_6}R10\\[0.5mm]
0  &0   &-x_4&0   &0   & 0&\frac{x_4}{x_3}R6 + R8+\frac{x_4}{x_6}R10\\
\end{array}\right)
}
\end{equation*}
\endgroup
$V_3$ is $9 \times 10$ and $C_3$ is $9 \times 6$. After the third column we have the following,
\begingroup
\renewcommand*{\arraystretch}{1}
\begin{equation*}
\scalemath{.66}{
\left(\begin{array}{cccccccccc}
 1&0  &0  &0  &0  &0 &0&0&0&0   \\[0.5mm]
  0& 1&0  &0  &0  &0 &0&0&0&0   \\[0.5mm]
 0 &  0& 0 & \frac{x_4}{x_1} & 0 & 0&1&0&\frac{x_4}{x_5}&\frac{x_4}{x_6}  \\[0.5mm]
 0 & 0 & 0 & \frac{x_4}{x_1} &  0& 0&0&1&0&\frac{x_4}{x_6}  \\[0.5mm]
 0 & 0 &  0& 0 & \frac{x_4}{x_2} & 0&1&0&\frac{x_4}{x_5}&\frac{x_4}{x_6}  \\[0.5mm]
 0 &  0& 0 & 0 & \frac{x_4}{x_2} & 0&0&1&0&\frac{x_4}{x_6}  \\[0.5mm]
  0&  0&x_4  & 0 & 0 &\frac{x_4}{x_3} &1&0&\frac{x_4}{x_5}&\frac{x_4}{x_6}  \\[0.5mm]
  0&  0&x_4  &  0& 0 &\frac{x_4}{x_3} &0&1&0&\frac{x_4}{x_6}  \\[0.5mm]
\end{array}\right)\cdot S^T = \left(\begin{array}{ccccccc}
1  &0   &0   &0   &0   & 0&R1\\[0.5mm]
0  &1   &0   &0   &0   & 0&R2\\[0.5mm]
-x_4&0   &0   &0   &0   & 0&\frac{x_4}{x_1}R4 + R7 +\frac{x_4}{x_5}R9+\frac{x_4}{x_6}R10\\[0.5mm]
-x_4&0   &0   &0   &0   & 0&\frac{x_4}{x_1}R4 + R8+\frac{x_4}{x_6}R10\\[0.5mm]
0  &-x_4&0   &0   &0   & 0&\frac{x_4}{x_2}R5 + R7 +\frac{x_4}{x_5}R9+\frac{x_4}{x_6}R10\\[0.5mm]
0  &-x_4&0   &0   &0   & 0&\frac{x_4}{x_2}R5 + R8+\frac{x_4}{x_6}R10\\[0.5mm]
0  &0   &0&0   &0   & 0&x_4R3+\frac{x_4}{x_3}R6 + R7 +\frac{x_4}{x_5}R9+\frac{x_4}{x_6}R10\\[0.5mm]
0  &0   &0&0   &0   & 0&x_4R3+\frac{x_4}{x_3}R6 + R8+\frac{x_4}{x_6}R10\\
\end{array}\right)
}
\end{equation*}
\endgroup
$V_4$ is $8 \times 10$ and $C_4$ is $8 \times 6$.  After the second column we have the following,
\begingroup
\renewcommand*{\arraystretch}{1}
\begin{equation*}
\scalemath{.68}{
\left(\begin{array}{cccccccccc}
1&0  &0  &0  &0  &0 &0&0&0&0   \\[0.5mm]
 0 &  0& 0 & \frac{x_4}{x_1} & 0 & 0&1&0&\frac{x_4}{x_5}&\frac{x_4}{x_6}  \\[0.5mm]
 0 & 0 &  0& \frac{x_4}{x_1} & 0 & 0&0&1&0&\frac{x_4}{x_6}  \\[0.5mm]
 0 &x_4  &  0&  0& \frac{x_4}{x_2} & 0&1&0&\frac{x_4}{x_5}&\frac{x_4}{x_6}  \\[0.5mm]
 0 &x_4  &  0& 0 & \frac{x_4}{x_2} & 0&0&1&0&\frac{x_4}{x_6}  \\[0.5mm]
  0& 0 &x_4  &  0&  0&\frac{x_4}{x_3} &1&0&\frac{x_4}{x_5}&\frac{x_4}{x_6}  \\[0.5mm]
  0& 0 &x_4  &  0&  0&\frac{x_4}{x_3} &0&1&0&\frac{x_4}{x_6}  \\[0.5mm]
\end{array}\right)\cdot S^T = \left(\begin{array}{ccccccc}
1  &0   &0   &0   &0   & 0&R1\\
-x_4&0   &0   &0   &0   & 0&\frac{x_4}{x_1}R4 + R7 +\frac{x_4}{x_5}R9+\frac{x_4}{x_6}R10\\[0.5mm]
-x_4&0   &0   &0   &0   & 0&\frac{x_4}{x_1}R4 + R8+\frac{x_4}{x_6}R10\\[0.5mm]
0  &0&0   &0   &0   & 0&x_4R2+\frac{x_4}{x_2}R5 + R7 +\frac{x_4}{x_5}R9+\frac{x_4}{x_6}R10\\[0.5mm]
0  &0&0   &0   &0   & 0&x_4R2+\frac{x_4}{x_2}R5 + R8+\frac{x_4}{x_6}R10\\[0.5mm]
0  &0   &0&0   &0   & 0&x_4R3+\frac{x_4}{x_3}R6 + R7 +\frac{x_4}{x_5}R9+\frac{x_4}{x_6}R10\\[0.5mm]
0  &0   &0&0   &0   & 0&x_4R3+\frac{x_4}{x_3}R6 + R8+\frac{x_4}{x_6}R10\\

\end{array}\right)
}
\end{equation*}
\endgroup
$V_5$ is $7 \times 10$ and $C_5$ is $7\times 6$. And finally,
\begingroup
\renewcommand*{\arraystretch}{1}
\begin{equation*}
\scalemath{.69}{
\left(\begin{array}{cccccccccc}
 x_4 & 0 &  0& \frac{x_4}{x_1} & 0 & 0&1&0&\frac{x_4}{x_5}&\frac{x_4}{x_6}  \\[1mm]
 x_4 &  0& 0 & \frac{x_4}{x_1} &  0& 0&0&1&0&\frac{x_4}{x_6}  \\[1mm]
 0 &x_4  & 0 &  0& \frac{x_4}{x_2} & 0&1&0&\frac{x_4}{x_5}&\frac{x_4}{x_6}  \\[1mm]
 0 &x_4  & 0 & 0 & \frac{x_4}{x_2} & 0&0&1&0&\frac{x_4}{x_6}  \\[1mm]
 0 &  0&x_4  & 0 & 0 &\frac{x_4}{x_3} &1&0&\frac{x_4}{x_5}&\frac{x_4}{x_6}  \\[1mm]
  0&  0&x_4  & 0 & 0 &\frac{x_4}{x_3} &0&1&0&\frac{x_4}{x_6}  \\[1mm]
\end{array}\right)\cdot S^T = \left(\begin{array}{ccccccc}
0&0   &0   &0   &0   & 0&x_4R1+\frac{x_4}{x_1}R4 + R7 +\frac{x_4}{x_5}R9+\frac{x_4}{x_6}R10\\[1mm]
0&0   &0   &0   &0   & 0&x_4R1+\frac{x_4}{x_1}R4 + R8+\frac{x_4}{x_6}R10\\[1mm]
0  &0&0   &0   &0   & 0&x_4R2+\frac{x_4}{x_2}R5 + R7 +\frac{x_4}{x_5}R9+\frac{x_4}{x_6}R10\\[1mm]
0  &0&0   &0   &0   & 0&x_4R2+\frac{x_4}{x_2}R5 + R8+\frac{x_4}{x_6}R10\\[1mm]
0  &0   &0&0   &0   & 0&x_4R3+\frac{x_4}{x_3}R6 + R7 +\frac{x_4}{x_5}R9+\frac{x_4}{x_6}R10\\[1mm]
0  &0   &0&0   &0   & 0&x_4R3+\frac{x_4}{x_3}R6 + R8+\frac{x_4}{x_6}R10\\
\end{array}\right)
}
\end{equation*}
\endgroup
$V_6$ is $6\times 10$ and $C_6$ is $6 \times 6$. The rows of $V_6$ are the basis for $\mathcal{N}(S(x))$.  While the $\mathrm{dim}\mathcal{N}(S(x))=4$, there are $6$ extreme pathway vectors.  These vectors are positive, and positively linearly independent,
\begingroup
\renewcommand*{\arraystretch}{1}
\begin{equation*}
\scalemath{.88}{
V_6=\left(\begin{array}{cccccccccc}
 x_4 &0  &0  & \frac{x_4}{x_1} &0  &0 &1&0&\frac{x_4}{x_5}&\frac{x_4}{x_6}\\[0.5mm]
 x_4 &0  &0  & \frac{x_4}{x_1} &  0& 0&0&1&0&\frac{x_4}{x_6}  \\[0.5mm]
0  &x_4  &0  &0  & \frac{x_4}{x_2} &0 &1&0&\frac{x_4}{x_5}&\frac{x_4}{x_6} \\[0.5mm]
0  &x_4  &0  &0  & \frac{x_4}{x_2} &0 &0&1&0&\frac{x_4}{x_6} \\[0.5mm]
0  &0  &x_4  &0  &0  &\frac{x_4}{x_3} &1&0&\frac{x_4}{x_5}&\frac{x_4}{x_6}\\[0.5mm]
0  &0  &x_4  &0  &0  &\frac{x_4}{x_3} &0&1&0&\frac{x_4}{x_6} \\
\end{array}\right)}.
\end{equation*}
\endgroup

\vspace*{4pt}\noindent\textbf{Example \ref{ex:compare} continued.}
The positive basis of $S$ was found in example \ref{ex:extremes}. This basis has six vectors despite the four dimensional nullspace. Also note that the basis vectors are not linearly independent, but positively linearly independent. The span of the positive basis vectors is shown below,
\begin{equation}\label{convex}
b_1
\begin{pmatrix}
x_6  \\
0  \\
0 \\
\frac{x_6}{x_1} \\
0 \\
0 \\
0 \\
\frac{x_6}{x_4} \\
0 \\
1 \\
\end{pmatrix}
+
b_2
\begin{pmatrix}
x_6  \\
0  \\
0 \\
\frac{x_6}{x_1} \\
0 \\
0 \\
\frac{x_6}{x_4} \\
0 \\
\frac{x_6}{x_5} \\
1 \\
\end{pmatrix}
+
b_3
\begin{pmatrix}
0  \\
x_6  \\
0 \\
0 \\
\frac{x_6}{x_2} \\
0 \\
0 \\
\frac{x_6}{x_4} \\
0 \\
1 \\
\end{pmatrix}
+
b_4
\begin{pmatrix}
0  \\
x_6  \\
0 \\
0 \\
\frac{x_6}{x_2} \\
0 \\
\frac{x_6}{x_4} \\
0 \\
\frac{x_6}{x_5} \\
1 \\
\end{pmatrix}
+
b_5
\begin{pmatrix}
0  \\
0  \\
x_6 \\
0 \\
0 \\
\frac{x_6}{x_3} \\
0 \\
\frac{x_6}{x_4} \\
0 \\
1 \\
\end{pmatrix}
+
b_6
\begin{pmatrix}
0  \\
0  \\
x_6 \\
0 \\
0 \\
\frac{x_6}{x_3} \\
\frac{x_6}{x_4} \\
0 \\
\frac{x_6}{x_5} \\
1 \\
\end{pmatrix}.
\end{equation}

Next the positive basis vectors for the nullspace must be intersected with the positive orthant $\mathcal{N}(S(x)) \cap (\mathbb{R}_+)^{m}$.
When this span is intersected with the positive orthant it is clear there is only one condition,
\begin{equation}\label{bs}
    \text{For all i, } b_i \geq 0 \qquad b_i\in\mathbb{R}.
\end{equation}

We refer to this span intersected with the positive orthant as $B$.

For the stoichiometric matrix $S(x)$ the basis for the null space was found. The span of the four basis vectors is shown below,
\begin{equation}\label{standard}
a_1
\begin{pmatrix}
x_6  \\
0  \\
0 \\
\frac{x_6}{x_1} \\
0 \\
0 \\
0 \\
\frac{x_6}{x_4} \\
0 \\
1 \\
\end{pmatrix}
+
a_2
\begin{pmatrix}
0  \\
0  \\
0 \\
0 \\
0 \\
0 \\
\frac{x_5}{x_4} \\
-\frac{x_5}{x_4} \\
1 \\
0 \\
\end{pmatrix}
+
a_3
\begin{pmatrix}
-x_3  \\
0  \\
x_3 \\
-\frac{x_3}{x_1} \\
0 \\
1 \\
0 \\
0 \\
0 \\
0 \\
\end{pmatrix}
+
a_4
\begin{pmatrix}
-x_2  \\
x_2  \\
0 \\
-\frac{x_2}{x_1} \\
1 \\
0 \\
0 \\
0 \\
0 \\
0 \\
\end{pmatrix}.
\end{equation}

To find the intersection of this span with the positive orthant, three conditions on the $a_i$ must hold.
\begin{equation}\label{as}
\begin{cases}
     a_i\geq 0, \text{ \hspace{2cm} for } i \in \{1,2,3,4\}\\
     a_1x_6 \geq a_3x_3 +a_4x_2,\\
     a_1x_6 \geq a_2x_5. \\
\end{cases}
\end{equation}

We refer to the intersection of this span with the positive orthant as $C$. We show that under the conditions given, $B = C$.

First it is shown that $B\subset C$ by showing that for an arbitrary set of $b_i \geq 0$, $a_i$'s can be chosen to reach the same vector.  Using the following substitutions for $a_i$ it is clear that if $b_i \geq 0$ the inequalities of \eqref{as} are satisfied
\begin{equation}\label{asubs}
    \begin{cases}
    a_1 = b_1+b_2+b_3+b_4+b_5+b_6,\\
    a_2 = (b_2+b_4+b_6)\frac{x_6}{x_5},\\
    a_3 = (b_5+b_6)\frac{x_6}{x_3},\\
    a_4 = (b_3+b_4) \frac{x_6}{x_2}.\\
    \end{cases}
\end{equation}

This shows that any vector in $B$ can be represented by vectors in $C$ i.e. $B\subset C$.

Now it is shown that $C\subset B$. For arbitrary $a$'s which satisfy \eqref{as}, the following substitutions for $b$ are used and the conditions of \eqref{bs} are checked.
\begin{equation}\label{bsubs}
    \begin{cases}
    b_1 = a_1 - a_3\frac{x_3}{x_6} - a_4\frac{x_2}{x_6} - a_2\frac{x_5}{x_6}+ b_4 +b_6,\\
    b_2 = a_2\frac{x_5}{x_6} - b_4 - b_6,\\
    b_3 = a_4\frac{x_2}{x_6} - b_4,\\
    b_5 = a_3\frac{x_3}{x_6} - b_6.\\
    \end{cases}
\end{equation}

To insure that \eqref{bs} is satisfied the following inequalities must hold,
\begin{align}
a_1 + b_4 +b_6&\geq a_3\frac{x_3}{x_6} + a_4\frac{x_2}{x_6} + a_2\frac{x_5}{x_6} \label{ineq1}, \\
a_2\frac{x_5}{x_6} &\geq b_4 + b_6  \label{ineq2}, \\
a_4\frac{x_2}{x_6} &\geq b_4  \label{ineq3}, \\
a_3\frac{x_3}{x_6} &\geq b_6 \label{ineq4}.
\end{align}

If choices for $b_4$ and $b_6$ can be found which satisfy these inequalities then we will have $C\subset B$.

We have the following two conditions:

\vspace*{4pt}\noindent\textbf{Condition 1.}
$a_4\frac{x_2}{x_6} + a_3\frac{x_3}{x_6} \leq  a_2\frac{x_5}{x_6}$.
Setting $b_4 = a_4\frac{x_2}{x_6}$ and $b_6 = a_3\frac{x_3}{x_6}$ immediately satisfies \eqref{ineq2}, \eqref{ineq3}, and \eqref{ineq4}. From \eqref{as} we have $a_1x_6 \geq a_2x_5$ which means that inequality \eqref{ineq1} is satisfied.

\vspace*{4pt}\noindent\textbf{Condition 2.}
$a_4\frac{x_2}{x_6} + a_3\frac{x_3}{x_6} > a_2\frac{x_5}{x_6}$.

Under Condition 2 three cases must be considered.

\vspace*{4pt}\noindent\textbf{Case 1.} $a_4\frac{x_2}{x_6} < a_2\frac{x_5}{x_6}$.  In this case we set $b_4 = a_4\frac{x_2}{x_6}$ and $b_6 = a_2\frac{x_5}{x_6} - a_4\frac{x_2}{x_6}$. This immediately satisfies \eqref{ineq2} and \eqref{ineq3}. Since $a_4\frac{x_2}{x_6} + a_3\frac{x_3}{x_6} > a_2\frac{x_5}{x_6}$  \eqref{ineq4} is satisfied as well. Then from \eqref{as} we have $a_1x_6 \geq a_3x_3 + a_4x_2$, which means that \eqref{ineq1} is satisfied.

\vspace*{4pt}\noindent\textbf{Case 2.} $a_4\frac{x_2}{x_6} > a_2\frac{x_5}{x_6}$ and $a_3\frac{x_3}{x_6} < a_2\frac{x_5}{x_6}$. In this case we set $b_6 = a_3\frac{x_3}{x_6}$ and $b_4 = a_2\frac{x_5}{x_6} - a_3\frac{x_3}{x_6}$, this satisfies \eqref{ineq2}, \eqref{ineq3}, \eqref{ineq4}. And from \eqref{as} we have $a_1x_6 \geq a_3x_3 + a_4x_2$, which means that \eqref{ineq1} is also satisfied.

\vspace*{4pt}\noindent\textbf{Case 3.} $a_4\frac{x_2}{x_6} > a_2\frac{x_5}{x_6}$ and $a_3\frac{x_3}{x_6} > a_2\frac{x_5}{x_6}$. In this case we set $b_4 = b_6 = \frac{1}{2}a_2\frac{x_5}{x_6}$ which satisfies \eqref{ineq2}, \eqref{ineq3}, \eqref{ineq4}. Then from \eqref{as} we have $a_1x_6 \geq a_3x_3 + a_4x_2$, which means that \eqref{ineq1} is also satisfied.

Since the $a's$ were arbitrary and $b's$ are found which satisfy $\eqref{bs}$ this gives us that  $C\subset B$ as desired.

\section{Acknowledgments}
The authors acknowledge the support of the Joseph and Loretta Lopez Chair Professorship endowment,
Sanofi via the project ``Optimization and Simulation Approaches or Systems Pharmacology in the Pharmaceutical Industry"
and the NSF Grant \# 1107444 KI-Net ``Kinetic description of emerging challenges in multiscale problems of natural sciences".

\renewcommand{\refname}{REFERENCES}

\medskip
Received for publication June  2018.
\medskip
\end{document}